\documentclass[reqno, 12pt]{amsart}
\pdfoutput=1
\makeatletter
\let\origsection=\section \def\section{\@ifstar{\origsection*}{\mysection}} 
\def\mysection{\@startsection{section}{1}\z@{.7\linespacing\@plus\linespacing}{.5\linespacing}{\normalfont\scshape\centering\S}}
\makeatother        

\usepackage{amsmath,amssymb,amsthm}
\usepackage{mathrsfs}
\usepackage{mathabx}\changenotsign
\usepackage{dsfont}
 
\usepackage{xcolor}
\usepackage[backref]{hyperref}
\hypersetup{
    colorlinks,
    linkcolor={red!60!black},
    citecolor={green!60!black},
    urlcolor={blue!60!black}
}

\usepackage[open,openlevel=2,atend]{bookmark}

\usepackage[abbrev,msc-links,backrefs]{amsrefs} 
\usepackage{doi}

\renewcommand{\PrintDOI}[1]{\doi{#1}}

\usepackage[T1]{fontenc}
\usepackage{lmodern}
\usepackage[babel]{microtype}

\usepackage[english]{babel}

\usepackage{geometry}
\usepackage{enumitem}
\def\rmlabel{\upshape({\itshape \roman*\,})}

\def\alabel{\upshape({\itshape \alph*\,})}
\def\Alabel{\upshape({\itshape \Alph*\,})}

\def\AAlabel{\upshape({A\arabic*})}
\def\Blabel{\upshape({B\arabic*})}
\def\Nlabel{\upshape({N\arabic*})}
\def\Clabel{\upshape({C\arabic*})}
\def\Flabel{\upshape({F\arabic*})}

\def\Rlabel{\upshape({R\arabic*})}

\def\Zlabel{\upshape({Z\arabic*})}
\def\Xilabel{\upshape({\,\itshape $\Xi$\arabic*\,})}

\def\Ilabelb{\upshape({I})}
\def\Tlabelb{\upshape({T})}
\def\Zlabelb{\upshape({Z})}

\let\polishlcross=\l
\def\l{\ifmmode\ell\else\polishlcross\fi}

\def\tand{\ \text{and}\ }
\def\qand{\quad\text{and}\quad}
\def\qqand{\qquad\text{and}\qquad}

\let\emptyset=\varnothing
\let\setminus=\smallsetminus

\makeatletter
\def\moverlay{\mathpalette\mov@rlay}
\def\mov@rlay#1#2{\leavevmode\vtop{   \baselineskip\z@skip \lineskiplimit-\maxdimen
   \ialign{\hfil$\m@th#1##$\hfil\cr#2\crcr}}}
\newcommand{\charfusion}[3][\mathord]{
    #1{\ifx#1\mathop\vphantom{#2}\fi
        \mathpalette\mov@rlay{#2\cr#3}
      }
    \ifx#1\mathop\expandafter\displaylimits\fi}
\makeatother

\newcommand{\dcup}{\charfusion[\mathbin]{\cup}{\cdot}}

\let\vphi=\varphi
\let\eps=\varepsilon
\let\theta=\vartheta
\let\rho=\varrho

\DeclareMathOperator{\Base}{Base} 
\DeclareMathOperator{\dens}{dens}
\DeclareMathOperator{\mad}{mad}

\def\tand{\ \text{and}\ }
\def\qand{\quad\text{and}\quad}
\def\qqand{\qquad\text{and}\qquad}

\def\NN{\mathds N}
\def\ZZ{\mathds Z}

\def\P{\mathds P}
\def\E{\mathds E}
\def\cA{{\mathcal A}}
\def\cH{{\mathcal H}}

\def\cF{{\mathcal F}}
\def\cG{{\mathcal G}}

\def\cC{{\mathcal C}}

\def\cJ{{\mathcal J}}

\def\cP{{\mathcal P}}

\def\cS{{\mathcal S}}

\def\cW{{\mathcal W}}

\def\ZZn{\ZZ/n\ZZ}

\def\talpha{\widetilde{\alpha}}

\def\Var{\text{Var}}

\theoremstyle{plain}
\newtheorem{satz}{Satz} 
\newtheorem{fact}[satz]{Fact}
\newtheorem{cor}[satz]{Corollary}
\newtheorem{definition}[satz]{Definition}
\newtheorem{lemma}[satz]{Lemma}

\newtheorem{theorem}[satz]{Theorem}

\newtheoremstyle{note}
  {4pt}
  {4pt}
  {\sl}
  {}
  {\itshape}
  {.}
  {.5em}
  {}

\theoremstyle{note}

\DeclareFontFamily{U}  {MnSymbolC}{}
\DeclareSymbolFont{MnSyC}         {U}  {MnSymbolC}{m}{n}
\DeclareFontShape{U}{MnSymbolC}{m}{n}{
    <-6>  MnSymbolC5
   <6-7>  MnSymbolC6
   <7-8>  MnSymbolC7
   <8-9>  MnSymbolC8
   <9-10> MnSymbolC9
  <10-12> MnSymbolC10
  <12->   MnSymbolC12}{}
\DeclareMathSymbol{\powerset}{\mathord}{MnSyC}{180}

\begin{document}
\title[Sharp thresholds for Ramsey properties of some nearly bipartite graphs]{Sharp thresholds for Ramsey properties of strictly balanced nearly bipartite graphs}

\author{Mathias Schacht}
\address{Fachbereich Mathematik, Universit\"at Hamburg, Hamburg, Germany}
\email{schacht@math.uni-hamburg.de}
\thanks{The first author was supported through the \emph{Heisenberg-Programme} of the DFG\@.}

\author{Fabian Schulenburg}
\email{fabian.schulenburg@uni-hamburg.de}


\date{\today}

\begin{abstract}
For a given graph~$F$ we consider the family of (finite) graphs~$G$ with the \emph{Ramsey property} for~$F$, that is the set of such graphs~$G$ with the property that every two-colouring of the edges of~$G$ yields a monochromatic copy of~$F$. 
For~$F$ being a triangle Friedgut, R\"odl, Ruci\'{n}ski, and Tetali~(2004) established the sharp threshold for the Ramsey property in random graphs.
We obtained a simpler proof of this result which extends to a more general class of graphs~$F$ including all cycles.

The proof is based on Friedgut's criteria~(1999) for sharp thresholds, and on the recently developed \emph{container method} for independent sets in hypergraphs by Saxton and Thomason, and Balogh, Morris and Samotij. The proof builds on some recent work of Friedgut et al.\ who established a similar result for van der Waerden's theorem.
\end{abstract} 
\maketitle

\section{Introduction}

A common theme in extremal and probabilistic combinatorics in recent years concerns the transfer of classical results to sparse random structures. Prime examples include Ramsey's theorem, Tur\'{a}n's theorem,  
and Szemer\'{e}di's theorem (see, e.g.,~\cites{RR95,CG10,FrRoSch10,Sch09}).
Here we often  want to replace the complete graph~$K_n$ or the set of integers $[n]=\{1,\dots,n\}$
(implicitly appearing in the classical results mentioned above) 
by a random graph $G(n,p)$ or a random subset of $[n]$. 

For example, in the context of Ramsey's theorem
for a given number of colours~$k$ and a graph $F$, one may consider the class $\cA$ of all 
graphs $G$ with the property that every $k$-colouring of its edges yields a monochromatic copy of $F$.
This leads to the following question: 
\textsl{When does the binomial random graph $G(n,p)$ satisfy~$\cA$ asymptotically almost surely (a.a.s.)?}
More precisely, for which~$p=p(n)$ we have $\lim_{n\rightarrow\infty}\P(G(n,p)\in\cA)=1$? It turns out that for many natural graph properties there exists a threshold function $\hat{p}=\hat{p}(n)$ such that
\begin{equation}\label{eq:1}
\lim_{n\rightarrow \infty}\P(G(n,p)\in \cA)= \begin{cases}
0, & \mathrm{if}\; p=o(\hat{p}), \\
1, & \mathrm{if}\; p=\omega(\hat{p}).
\end{cases}
\end{equation}

After establishing a threshold for a given property~$\cA$, one may study more closely how quickly the transition from a.a.s.\ not having~$\cA$ to a.a.s.\ having~$\cA$ occurs. If 
one can replace $p=o(\hat{p})$ in~\eqref{eq:1} by $p\leq(1-\eps)\hat{p}$ for every $\eps>0$ and similarly $p=\omega(\hat{p})$ can be replaced by~$p\geq (1+\eps)\hat{p}$, then the threshold is \textit{sharp} and otherwise  it is \textit{coarse}.

In that direction only a few results are known. In~\cite{FrBou99} Friedgut presents a characterization of coarse thresholds in a general setting. In case of random graphs it roughly says that a  threshold is coarse if and only if it is correlated to a local property. For example the graph property ``$G(n,p)$ contains a triangle'' depends on local events and has a coarse threshold while the graph property ``$G(n,p)$ is connected'' is a global property and has in fact a sharp threshold.

Friedgut's work yields a tool to verify sharp thresholds by contradiction. Supposing to the contrary that
 the threshold in question would be coarse, one may use the characterization of Friedgut to deduce additional 
 structural properties (see e.g.\ Theorem~\ref{thm:Bou}) which might be  used to derive a contradiction.

There are some results in this area based on this approach. For example, it was shown in~\cite{FrKr00} that the Ramsey-type property ``in every vertex colouring of $G(n,p)$ with two colours there is a monochromatic triangle'' has a sharp threshold (see~\cite{FrKr00} for some related results).

Regarding Ramsey-type properties concerning edge colourings the applicability of Friedgut's criterion seems more involved. In that direction it was shown by  Friedgut, R\"odl, Ruci\'{n}ski, and Tetali in~\cite{FRRT06} that the Ramsey property for the triangle and two colours has a sharp threshold. More recently, Friedgut, H\`{a}n, Person and Schacht~\cite{FHPS14} studied van der Waerden's property in random subsets of~$\ZZ/n\ZZ$ and established a sharp threshold for this property.
Essentially the same proof yields the sharpness of the threshold 
for the Ramsey properties of strictly balanced (see~\eqref{eq:m2-def} below) $k$-partite $k$-uniform hypergraphs and, 
hence, in particular for even cycles in graphs and two colours.

We extend this research to non-bipartite graphs. In particular, we obtain a shorter proof of the triangle result from~\cite{FRRT06}. We will use the arrow notation from Ramsey theory. For two graphs~$G$ and~$F$ we write $G\rightarrow (F)_r^e$ if for all edge colourings of~$G$ with~$r$ colours there exists a monochromatic copy of~$F$. If, on the other hand, there
is an $r$-colouring of $E(G)$ with no monochromatic copy of $F$, then we write $G\nrightarrow (F)_r^e$.
 Our first result establishes the sharp threshold when $F$ is a 
cycle.

\begin{theorem}\label{thm:MainCycle}
For a cycle~$C_k$ of length $k\geq 3$ there exist positive constants~$c_0$ and~$c_1$ and a function~$c(n)$ with $c_0<c(n)<c_1$ such that for all $\eps >0$ we have
\[\lim_{n\rightarrow \infty} \P(G(n,p)\rightarrow (C_k)_2^e)=
\begin{cases}
0, & \mathrm{if}\; p\leq (1-\eps)c(n)n^{-(k-2)/(k-1)}, \\
1, & \mathrm{if}\; p\geq (1+\eps)c(n)n^{-(k-2)/(k-1)}.
\end{cases}\]
\end{theorem}
We establish the sharpness of the threshold for Ramsey-properties of a more general class of graphs than cycles. For a graph $F=(V,E)$ we write $v(F)=\vert V(F)\vert $ and $e(F)=\vert E(F)\vert$. For graphs $F$ with at least one edge let the \textit{$2$-density} $m_2(F)$ be defined by
\[m_2(F)=\max\lbrace d_2(F^\prime)\colon F^\prime \subseteq F \text{ and } e(F^\prime)\geq 1 \rbrace\,,\]
where
\begin{equation}
	\label{eq:m2-def}
	d_2(F^\prime)=\begin{cases}
	\frac{e(F^\prime)-1}{v(F^\prime)-2},& \mathrm{if }\; v(F^\prime)>2,\\
	1,& \mathrm{if }\; F^\prime =K_2.\end{cases}
\end{equation}
If $d_2(F)=m_2(F)$, then~$F$ is \textit{balanced}. Moreover, $F$ is 
\textit{strictly balanced} if in addition $d_2(F^\prime)<m_2(F)$ for all proper subgraphs $F^\prime\subsetneq F$
with at least one edge. We say a graph~$F$ is \textit{nearly bipartite} if $e(F)\geq 2$ and 
there is a bipartite graph~$F^\prime$ and some edge $e$ such that $F=F^\prime+e =(V(F^\prime),E(F^\prime)\cup \lbrace e\rbrace)$. 
Note that this definition includes all bipartite graphs with at least two edges.
Since for every $k\geq 3$ the cycle~$C_k$ of length $k$
is strictly balanced and
nearly bipartite, the following result includes
Theorem~\ref{thm:MainCycle} as a special case.

\begin{theorem}\label{thm:Main}
For all strictly balanced and nearly bipartite graphs $F$
there exist positive constants~$c_0$ and~$c_1$ and a function~$c(n)$ with $c_0<c(n)<c_1$ such that for all $\eps >0$ we have
\[\lim_{n\rightarrow \infty} \P(G(n,p)\rightarrow (F)_2^e)=
\begin{cases}
0, & \mathrm{if}\; p\leq (1-\eps)c(n)n^{-1/m_2(F)}, \\
1, & \mathrm{if}\; p\geq (1+\eps)c(n)n^{-1/m_2(F)}.
\end{cases}\]
\end{theorem}
We remark that here we defined $d_2(K_2)=1$ in~\eqref{eq:m2-def}. As a consequence it follows that
\begin{equation}
	\label{eq:strict}
	m_2(F)>1
\end{equation}
for every strictly balanced and nearly bipartite graph $F$, since every nearly bipartite graph 
is required to have at least two edges by definition. Moreover, we remark that 
the assumptions of Theorem~\ref{thm:Main} are never met by forests~$F$ and for sharp thresholds of Ramsey properties of trees we refer to~\cite{FrKr00}.

The proof of Theorem~\ref{thm:Main} refines ideas from the work in~\cite{FHPS14} and also uses
Friedgut's criterion for coarse thresholds~\cite{FrBou99} and the recent \emph{hypergraph container} results of Balogh, Morris, and Samotij~\cite{BMS12} and Saxton and Thomason~\cite{ST12}. 
In Section~\ref{sec:tools and outline} we will introduce these tools and in addition
we will state the two main technical lemmas, Lemmas~\ref{lem:hauptlemma1} and~\ref{lem:base},
 which we will need in the proof of the main result. Section~\ref{sec:proof main} is devoted to the proof 
 of Theorem~\ref{thm:Main} based on these tools. 
 In Section~\ref{sec:proof Lemma 1} and Section~\ref{sec:proof Lemma 2} we then prove Lemmas~\ref{lem:hauptlemma1} and~\ref{lem:base}, respectively.
We close with a few remarks concerning possible generalisations of Theorem~\ref{thm:Main} and related open questions.

\section{Main tools and outline of the proof}\label{sec:tools and outline}

In this section we introduce the necessary tools for the proof of the main result. For definiteness 
we may assume that the vertex sets of $K_n$ and $G(n,p)$ coincide with $[n]$.
We use the following notation: For a graph~$B$ and $n\geq v(B)$ we define $\Psi_{B,n}$ as the set of all injective embeddings of~$B$ into the complete graph~$K_n$. So~$\Psi_{B,n}$ corresponds to the unlabelled copies of~$B$ in~$K_n$ and, clearly,
 $\vert \Psi_{B,n}\vert=\Theta(n^{v(B)})$.

The starting point of the proof is the R\"{o}dl-Ruci\'{n}ski theorem (stated below)
which establishes that $n^{-1/m_2(F)}$ is the threshold for the property $G(n,p)\rightarrow (F)_2^e$
for most graphs~$F$. 
In view of Theorem~\ref{thm:Main} we restrict our discussion below 
to two colours and to strictly balanced and nearly bipartite graphs~$F$. In particular, owing to~\eqref{eq:strict}
we have $m_2(F)>1$ and exclude all forests (some forests exhibit a slightly different behaviour in this context
see~\cite{JLR00}*{Theorem 8.1} for details).

\begin{theorem}[R\"{o}dl \& Ruci\'{n}ski (special case)]\label{thm:Coa}
For all strictly balanced and nearly bipartite graphs~$F$, the function $\hat{p}=\hat p(n)=n^{-1/m_2(F)}$ is the threshold for the property~$G(n,p)\rightarrow (F)_2^e$. In fact, there exist constants $C_1\geq C_0>0$ such that
\[\lim_{n\rightarrow \infty} \P(G(n,p)\rightarrow (F)_2^e)=\begin{cases}
0, & \mathrm{if}\; p\leq C_0 n^{-1/m_2(F)}, \\
1, & \mathrm{if}\; p\geq C_1n^{-1/m_2(F)}.
\end{cases}\]\qed
\end{theorem}

We will strengthen Theorem~\ref{thm:Coa} and show that these thresholds are sharp. For that we will appeal to Friedgut's  criterion for coarse thresholds which will be introduced in Section~\ref{subsec:Bou}. Then we present a recent structural result on independent sets in hypergraphs which plays a crucial r\^ole in our proof. In Section~\ref{subsec:mainlemmas} we introduce two somewhat technical probabilistic lemmas needed for the proof of Theorem~\ref{thm:Main}. Section~\ref{subsec:Colhitset} establishes the connection between independent sets in hypergraphs 
and colourings of the edges of the random graph without monochromatic copies of the given graph $F$ considered 
in our setting.

\subsection{Friedgut's criterion for coarse thresholds}\label{subsec:Bou}
The following characterisation of coarse thresholds appeared in~\cite{Fr05}*{Theorem~2.4}. 
\begin{theorem}\label{thm:Bou}
Let~$\cA$ be a monotone graph property with a coarse threshold. Then there exist $p=p(n)$, constants $\frac{1}{3}>\alpha>0$, $\eps>0$, $\tau>0$, and a graph~$B$ satisfying
\begin{enumerate}[label=\rmlabel]
\item\label{it:ThmBouVor1} $\alpha<\P(G(n,p)\in\cA)<1-3\alpha$ and
\item\label{it:ThmBouVor2} $\P(B\subseteq G(n,p))>\tau$
\end{enumerate}
such that for every graph property~$\cG$ with a.a.s.\ $G(n,p)\in\cG$ 
there exist infinitely many $n\in\NN$ and for each such~$n$ a graph $Z\in\cG$ on~$n$ vertices such that the following holds.
\begin{enumerate}
\item $\P(Z\cup h(B)\in\cA)>1-\alpha$\,, where $h\in \Psi_{B,n}$ is chosen uniformly at random,
\item $\P(Z\cup G(n,\eps p)\in\cA)<1-2\alpha$,
\end{enumerate}
where the random graph $G(n,\eps p)$ and~$Z$ have the same vertex set.\qed
\end{theorem}

Note that the $\P(\cdot)$ in~\ref{it:ThmBouVor1} (and~\ref{it:ThmBouVor2}), in~\eqref{it:Boo1}, and in~\eqref{it:Boo2} concern different probability spaces. While in~\ref{it:ThmBouVor1} and~\ref{it:ThmBouVor2} it concerns the random graph $G(n,p)$ we consider~$h$ chosen uniformly at random in~\eqref{it:Boo1} and the random graph $G(n,\eps p)$ in~\eqref{it:Boo2}. 
 Below we reformulate Theorem~\ref{thm:Bou} suited for our application.

\begin{cor}\label{cor:Boo}
Let~$F$ be a strictly balanced and nearly bipartite graph. Assume that the property $G\rightarrow (F)_2^e$ does not have a sharp threshold. Then there exists  a function $p(n)=c(n)n^{-1/m_2(F)}$ with $C_0<c(n)<C_1$ for some $C_0,C_1>0$, there are 
constants $\frac{1}{3}>\alpha>0$ and $\eps >0$, and there is a graph~$B$ with $B\nrightarrow (F)_2^e$ such that for infinitely many $n\in \NN$ and for every family of graphs~$\cG$  on~$n$ vertices with a.a.s.\  $G(n,p)\in \cG$ 
there exists a $Z\in\cG$ such that the following hold 
\begin{enumerate}
\item\label{it:Boo1} $\P(Z\cup h(B) \rightarrow (F)_2^e)>1-\alpha$, with $h\in \Psi_{B,n}$ chosen uniformly at random,
\item\label{it:Boo2} $\P(Z\cup G(n,\eps p) \rightarrow (F)_2^e)<1-2\alpha$.
\end{enumerate}
\end{cor}

Corollary~\ref{cor:Boo} is just a reformulation of Theorem~\ref{thm:Bou} in our context. We give the details below.

\begin{proof}[Proof of Corollary~\ref{cor:Boo}]
Note that conclusions~\eqref{it:Boo1} and~\eqref{it:Boo2} of Corollary~\ref{cor:Boo} are identical to~\eqref{it:Boo1} and~\eqref{it:Boo2} of Theorem~\ref{thm:Bou} for the monotone graph property $\cA=\lbrace G\colon G\rightarrow(F)_2^e\rbrace $. Owing to Theorem~\ref{thm:Coa} we infer that because of~\ref{it:ThmBouVor1} in Theorem~\ref{thm:Bou} the probability~$p(n)$ must satisfy $p(n)=c(n)n^{-1/m_2(F)}$ where $C_0<c(n)<C_1$ for constants $C_0,C_1$ given by Theorem~\ref{thm:Coa}. It is only left to show that $B\nrightarrow (F)_2^e$ is a consequence of~\ref{it:ThmBouVor2} of Theorem~\ref{thm:Bou}. 

Recall that it was shown in~\cite{RR93}*{Theorem~6} that if $B\rightarrow (F)_2^e$ then $m(B)=\frac{e(B)}{v(B)}>m_2(F)$. Thus a standard application of Markov's inequality yields $\P(H\subseteq G(n,p))=o(1)$ for every~$H$ with $H\rightarrow (F)_2^e$
and $p=\Theta(n^{-1/m_2(F)})$. Consequently the graph~$B$ provided by Theorem~\ref{thm:Bou} must satisfy $B\nrightarrow (F)_2^e$, due to~\ref{it:ThmBouVor2} of Theorem~\ref{thm:Bou}.
\end{proof}

\subsection{Hypergraph containers}\label{subsec:hypergraph containers}
We shall also use a recent result concerning independent sets in hypergraphs, which was obtained independently by Saxton and Thomason~\cite{ST12} and Balogh, Morris, and Samotij~\cite{BMS12}. Here we will use the version from~\cite{ST12}.

Let~$\cH$ be an $\l$-uniform hypergraph on~$m=|V(\cH)|$ vertices. 
For a subset $\sigma\subset V(\cH)$ we define its \emph{degree} by 
\[d(\sigma)=\vert \{e\in E(\cH)\colon\sigma\subseteq e\}\vert\,.\]
For a vertex $v\in V$ and an integer $j$ with $2\leq j\leq \l$ we consider the maximum degree 
over all $j$-element sets~$\sigma$ containing $v$
\[
	d^{(j)}(v)=\max\{d(\sigma)\colon  v\in \sigma\subset V(\cH) \tand \vert \sigma\vert =j\}\,.
\]
We denote by $d=\l|E(\cH)|/m>0$ the average degree of~$\cH$ and, following the notation of~\cite{ST12},
for $\tau>0$ and  $j=2,\dots,\l$ we set 
\[\delta_j=\frac{1}{\tau^{j-1}md}\sum_{v\in V(\cH)}d^{(j)}(v)\]
and
\[\delta(\cH,\tau)=2^{\binom{\l}{2}-1}\sum_{j=2}^\l 2^{-\binom{j-1}{2}}\delta_j\,.\]
We write~$\powerset(X)$ for the power set of~$X$ and denote by~$\powerset^s(X)=\powerset(X)\times\dots\times\powerset(X)$ the $s$-fold cross product of~$\powerset(X)$.

\begin{theorem}[Saxton \& Thomason]\label{thm:ST}
Let~$\cH$ be an $\l$-uniform hypergraph on the vertex set~$[m]$ and let $0<\eps<\tfrac{1}{2}$. Suppose that for~$\tau>0$ we have
$\delta(\cH,\tau)\leq \eps/12\l!$ and $\tau\leq 1/144\l!^2\l$. Then there exist a constant $c=c(\l)$ and a collection $\cJ\subset \powerset([m])$ such that the following holds
\begin{enumerate}[label=\alabel]
\item\label{it:ST1} for every independent set~$I$ in~$\cH$ there exists $T=(T_1,\dots,T_s)\in \powerset^s(I)$ with $\vert T_i\vert\leq c\tau m$, $s\leq c\log(1/\eps)$  and there exists a~$J=J(T)\in\cJ$ only depending on~$T$ such that $I\subseteq J(T)\in\cJ$,
\item\label{it:ST2} $e(\cH[J])\leq \eps e(\cH)$ for all $J\in\cJ$ and
\item\label{it:ST3} $\log \vert \cJ\vert \leq c\tau\log(1/\tau)\log(1/\eps)m$.\qed
\end{enumerate}
\end{theorem}

We will apply Theorem~\ref{thm:ST} to an auxiliary hypergraph described in the following section.

\subsection{Main probabilistic lemmas}\label{subsec:mainlemmas}

The hypergraph~$\cH$ to which we will apply Theorem~\ref{thm:ST} depends on the graph $Z\in \cG$ which will 
be provided by Friedgut's criterion (Corollary~\ref{cor:Boo}) applied for the strictly balanced, nearly bipartite graph~$F$. For the verification of the assumptions of Theorem~\ref{thm:ST} we will restrict the family~$\cG$ containing~$Z$. Recall that~$\cG$ can be chosen to be any graph property which is satisfied a.a.s.\ by $G(n,p)$ for every $p$ with $p=\Theta(n^{-1/m_2(F)})$. In what follows we discuss the restrictions for the family~$\cG$ (see Lemmas~\ref{lem:hauptlemma1} and~\ref{lem:base} below) and for that we introduce the required notation.

Let~$Z$ and~$B$ be two subgraphs of the complete graph~$K_n$. We say $z\in E(Z)$ \textit{focuses} on $b\in E(B)$ if 
there exists a copy of $F$ in $Z\cup B$ which contains $z$ and $b$.
We set
\begin{equation}\label{def:MZB}
	M(Z,B)=\lbrace z\in E(Z)\colon \text{there is }b\in E(B)\text{ such that }z	\text{ focuses on }b \rbrace\,.
\end{equation}
The pair $(Z,B)$ is called \textit{interactive} if $E(Z)\cap E(B)=\emptyset$,
$Z\nrightarrow (F)_2^e$, and $B\nrightarrow (F)_2^e$, but $Z\cup B\rightarrow (F)_2^e$. For a collection $\Xi\subset \Psi_{B,n}$ of embeddings of~$B$ into~$K_n$ the pair $(Z,\Xi)$ is called \textit{interactive} if $(Z,h(B))$ is interactive for all $h\in \Xi$. Furthermore, a 
pair $(Z,\Xi)$ is \textit{regular} if
for all $h\in \Xi$ every $z\in E(Z)$ focuses on at most one $b\in E(h(B))$. 
We call~$h\in\Psi_{B,n}$ \emph{regular} w.r.t.\ $Z$ if $(Z,\{h\})$ is regular. 
The hypergraphs~$\cH$ considered here are defined in terms of regular pairs $(Z,\Xi)$.

For a pair $(Z,\Xi)$ with $Z\subseteq K_n$ and $\Xi\subseteq \Psi_{B,n}$
we define the hypergraph $\cH=\cH(Z,\Xi)$ with vertex set 
\[	
	V(\cH)=E(Z)
\] and edge set 
\[
	E(\cH)=\lbrace M(Z,h(B))\colon h\in \Xi\rbrace\,.
\]

For our presentation it will be useful to consider orderings of the edges of the involved graphs
and ``order consistent'' embeddings.
For that we fix an arbitrary ordering of $E(K_n)$ and an ordering of~$E(B)$.
For an interactive and regular pair $(Z,\Xi)$ and~$h\in \Xi$ we say that $z\in M(Z,h(B))=\lbrace e_1,\dots,e_\l\rbrace$ with $e_1<e_2<\dots <e_\l$ has \textit{index}~$i$ if $z=e_i$.
Furthermore, we call $(Z,\Xi)$ and $\cH(Z,\Xi)$ \textit{index consistent} if for all $z\in E(Z)$ and all $h,h^\prime\in \Xi$ with $z\in M(Z,h(B))\cap M(Z,h^\prime(B))$ the indices of~$z$ in~$M(Z,h(B))$ and in~$M(Z,h^\prime(B))$ are the same.
Let $b_1<\dots<b_{e(B)}$ be the ordering of the edges of~$B$. Then the \textit{profile} of~$M(Z,h(B))$ is the  function $\pi\colon[\vert M(Z,h(B)) \vert]\rightarrow [e(B)]$ defined by $\pi(i)=j$ if and only if~$e_i$ focuses on~$h(b_j)$.
Since the pair $(Z,\Xi)$ is regular, for each edge of~$\cH$ each~$e_i$ focuses on at most one~$h(b_j)$ and, hence, the profile is well defined.
We say $(Z,\Xi)$ has profile~$\pi$ if all edges~$M(Z,h(B))$ for $h\in \Xi$ have profile~$\pi$.
Note that in this case all sets~$M(Z,h(B))$ have the same cardinality
and $\vert M(Z,h(B))\vert$ is called the \textit{length} of the profile $\pi$.

Having established this notation we now state the following technical lemma which gives one part of the graph property~$\cG$ for the application of Corollary~\ref{cor:Boo}.
Moreover, we shall also apply Theorem~\ref{thm:ST} which results in useful properties of the hypergraph $\cH(Z,\Xi)$ 
for $Z\in\cG$ and some appropriately chosen $\Xi\subseteq \Psi_{B,n}$.

\begin{lemma}\label{lem:hauptlemma1}
For all constants $C_1>C_0>0$, $\frac{1}{3}>\alpha>0$ and graphs~$F$ and~$B$, where~$F$ is strictly balanced, there exist $\alpha^\prime, \beta, \gamma>0$ and $L\in \NN$ such that for every $p=c(n)n^{-1/m_2(F)}$ with $C_0\leq c(n)\leq C_1$ a.a.s.\ $Z\in G(n,p)$ satisfies the following.
If \[\P(Z\cup h(B)\rightarrow (F)_2^e)>1-\alpha\] then there exists $\Xi_{B,n}\subseteq \Psi_{B,n}$ with $\vert \Xi_{B,n}\vert \geq \alpha^\prime n^2$ and $Z\cup h(B)\rightarrow (F)_2^e$ for all $h\in \Xi_{B,n}$ such that the hypergraph $\cH=\cH(Z,\Xi_{B,n})$ is index consistent for some profile~$\pi$ of length $\l\leq L$ and there is a family~$\cC$ of subsets of~$V(\cH)$ satisfying
\begin{enumerate}
\item\label{it:lem:core1} $\log\vert \cC \vert \leq e(Z)^{1-\gamma}$,
\item\label{it:lem:core2} $\vert C \vert \geq \beta e(Z)$ for all $C\in \cC$ and
\item\label{it:lem:core3} every hitting set $A$ of~$\cH$ contains a $C\in \cC$, i.e., for every $A\subseteq V(\cH)$
with $e\cap A\neq\emptyset$ for all $e\in E(\cH)$ there exists $C\in\cC$ with $C\subseteq A$.
\end{enumerate}
\end{lemma}

Note that in contrast to the assumptions of Theorem~\ref{thm:Main}
for Lemma~\ref{lem:hauptlemma1} it is not required that the given graph~$F$ is nearly bipartite. However, for the proof of Theorem~\ref{thm:Main} we need another restriction on the family~$\cG$ (in Corollary~\ref{cor:Boo})
which is satisfied a.a.s.\ by $G(n,p)$ and makes use of the near-bipartiteness of~$F$. For a nearly bipartite graph $F=F'+e$
we consider those pairs of vertices in~$K_n$ which complete a copy of the bipartite subgraph~$F'$ in a given 
subgraph of $G(n,p)$ to a full copy of~$F$ in $K_n$.
Hence, for a graph~$G\subseteq K_n$ we define the \emph{basegraph} $\Base_F(G)\subseteq K_n$ 
with edge set
\[
	\big\{\{x,y\}\colon \exists F'\subseteq G\ \text{such that}\ F'+\{x,y\}\ \text{forms a copy of}\ F\big\}\,.
\]
We require that for every relatively dense subgraph $G'$ of  $G(n,p)$ the basegraph
spans many copies of~$F$ itself.
More precisely, for a graph~$G$ on $n$ vertices and a nearly bipartite graph~$F=F'+e$
and $\lambda,\eta>0$ we say~$G$ has the property $T(\lambda,\eta,F)$ if for every subgraph $G^\prime\subset G$ with $e(G^\prime)\geq \lambda e(G)$ we have that the basegraph $\Base_F(G^\prime)$ contains at least $\eta n^{v(F)}$ copies of~$F$.

Lemma~\ref{lem:base} gives the second restriction for the family~$\cG$ for our application of Corollary~\ref{cor:Boo}.
Assuming that there is no copy of~$F$ in the bigger colour class of~$Z$,
Lemma~\ref{lem:base} will be helpful to find a copy of~$F$ in the intersection of $Z\cap G(n,\eps p)$ with the other colour class.

\begin{lemma}\label{lem:base}
For all $\lambda>0$, $C_1>C_0>0$ and every strictly balanced and nearly bipartite graph~$F$ there exists $\eta>0$ such that for $C_0n^{-1/m_2(F)}\leq p\leq C_1n^{-1/m_2(F)}$ the random graph $G(n,p)$ a.a.s.\ satisfies $T(\lambda, \eta,F)$.
\end{lemma}

\subsection{Colourings and hitting sets}\label{subsec:Colhitset}

In this section we establish the connection between hitting sets of the hypergraph~$\cH(Z,\Xi)$ and $F$-free colourings of~$Z$.

Recall that the definition of an interactive pair $(Z,\Xi)$ 
says that for every embedding $h\in\Xi\subseteq \Psi_{B,n}$ 
 the graphs $Z$ and $h(B)$ are edge disjoint and
 $Z\nrightarrow (F)_2^e$ and $B\nrightarrow (F)_2^e$ but $Z\cup h(B) \rightarrow (F)_2^e$.
Let $b_1,\dots,b_K$ be an enumeration of $E(B)$ and fix an $F$-free colouring $\sigma\colon E(B)\rightarrow \{\text{red,blue}\}$.
We copy this colouring for every $h\in\Xi$ by setting $\sigma_h\colon E(h(B))\rightarrow \{\text{red,blue}\}$ with $\sigma_h(h(b_i))=\sigma(b_i)$ for all~$i=1,\dots,K$.
Furthermore, let~$\vphi$ be an arbitrary $F$-free colouring of~$Z$. 

Since $Z\cup h(B) \rightarrow (F)_2^e$, the joint colouring of $Z\cup h(B)$ given by~$\vphi$ and~$\sigma_h$
 yields a monochromatic copy of~$F$ and this copy
must contain edges of both graphs, of~$Z$ and of~$h(B)$.
Thus each edge $M(Z,h(B))$ of the hypergraph $\cH(Z,\Xi)$ 
contains an $e\in E(Z)$ which focuses on some~$h(b)$ with~$b\in E(B)$, where we have $\vphi(e)=\sigma_h(h(b))=\sigma(b)$.
We say such an edge $e\in E(Z)$ (resp.\ vertex $e\in V(\cH)$) is \textit{activated} by~$\vphi$,~$\sigma$, and~$h$.
We define the set of activated vertices by
\begin{equation}\label{eq:defA}
	A_\vphi^{\sigma}
	=
	A_\vphi^{\sigma}(Z,\Xi)=\bigcup_{h\in\Xi}\{e\in E(Z)\colon e\text{ is activated by $\sigma$, $\vphi$ and $h$}\}\subseteq V(\cH)\,.
\end{equation}
Note that by definition for an interactive pair $(Z,\Xi)$ every edge $M(Z,h(B))$ of $\cH(Z,\Xi)$ contains an activated vertex and, hence, the set of activated vertices~$A_\vphi^\sigma$ is  a hitting set of~$\cH(Z,\Xi)$.
In what follows we will use different colourings~$\vphi$ of~$Z$ but we will always restrict to the same colouring~$\sigma$ of~$B$.

Suppose that in addition we have a fixed ordering of~$E(Z)$ and $E(B)=\{b_1,\dots,b_K\}$. Further suppose  that
the interactive pair $(Z,\Xi)$ is also index consistent with profile~$\pi$ of length~$\l$. In particular, the hypergraph $\cH(Z,\Xi)$ is $\l$-uniform.

It also follows from the definitions that for $z\in A^\sigma_\vphi\cap A^\sigma_{\vphi^\prime}$ for two colourings~$\vphi$ and~$\vphi^\prime$ we have $\vphi(z)=\vphi^\prime(z)$.
In fact, for $z\in A^{\sigma}_\vphi$ there exists an $h\in\Xi$ such that~$z$ is activated by~$\sigma$,~$\vphi$ and~$h$.
Let~$i$ be the index of~$z$ in $M(Z,h(B))$, then~$z$ focuses on $h(b_{\pi(i)})$ and, therefore, $\vphi(z)=\sigma(b_{\pi(i)})$.
Repeating the same argument for~$\vphi^\prime$, we obtain from index consistency that
 $\vphi^\prime(z)=\sigma(b_{\pi(i)})=\vphi(z)$.
We summarise these observations in the following fact.

\begin{fact}\label{fact:hittingset}
Let $(Z,\Xi)$ be an interactive, regular and index consistent pair with profile~$\pi$ and 
let~$\sigma$ be an $F$-free colouring of~$E(B)$ and $\vphi$ be an $F$-free colouring of $E(Z)$. Then
\begin{enumerate}[label=\AAlabel]
	\item\label{it:hitting1} $A_\vphi^\sigma(Z,\Xi)$ is a hitting set of $\cH(Z,\Xi)$ and
	\item\label{it:hitting2} for all $F$-free colourings $\vphi^\prime$ of~$E(Z)$
		and for all $z\in A^{\sigma}_\vphi \cap A^{\sigma}_{\vphi^\prime}$ we have $\vphi(z)=\vphi^\prime(z)$.
\end{enumerate}

\end{fact}

Now we are prepared to give the proof of the main theorem based on the lemmas and theorems of this section.
\vspace{0.4mm}
\section{Proof of the main theorem}\label{sec:proof main}
The starting point of the proof is Friedgut's criterion (see Corollary~\ref{cor:Boo}) applied to the contradictory assumption, that the Ramsey property $G\rightarrow (F)_2^e$ for a given strictly balanced and nearly bipartite graph $F$
has a coarse threshold.
For that we define a family of graphs~$\cG$ having ``useful'' properties and Lemma~\ref{lem:hauptlemma1} and Lemma~\ref{lem:base} show that a.a.s.\ $G(n,p)$ displays these properties.
Then Friedgut's criterion asserts for infinitely many $n\in\NN$ the existence of an $n$-vertex graph $Z\in\cG$, a graph~$B$ (called \emph{booster}), constants $\frac{1}{3}>\alpha>0$, $\eps>0$ and a family of embeddings $\Psi_{B,n}^\prime\subseteq \Psi_{B,n}$ with $Z\cup h(B)\rightarrow (F)_2^e$ for all $h\in\Psi_{B,n}^\prime$ and $\vert \Psi_{B,n}^\prime\vert\geq (1-\alpha)\vert\Psi_{B,n}\vert$, but $\P(Z\cup G(n,\eps p)\rightarrow (F)_2^e)<1-2\alpha$.
The goal is to find a contradiction to the last fact by showing $\P(Z\cup G(n,\eps p)\rightarrow (F)_2^e)=1-o(1)$. 

Let~$\Phi$ be the set of all $F$-free colourings of~$Z$.
We have to show that for any $\vphi\in\Phi$ the probability to extend~$\vphi$ to an $F$-free colouring of $Z\cup G(n,\eps p)$ is very small.
We are able to show that this probability is of order $\exp(-\Omega (pn^2))$.
Now we would like to use a union bound for all $\vphi\in\Phi$.
However, we have only little control over $\vert \Phi\vert$ and the trivial upper bound $2^{\Theta(pn^2)}$ is too large to combine it with the bound from above~$\exp (-\Omega(pn^2))$ to obtain for $\P(Z\cup G(n,\eps p)\nrightarrow (F)_2^e)$ a bound of order~$o(1)$ by the union bound. 

Instead we shall find a partition of~$\Phi$ into $2^{o(pn^2)}$ classes such that two colourings from the same partition class always agree on a large subset of~$Z$.
These subsets are called \textit{cores}.
Then we will show that the colouring of~$\vphi$ restricted to the associated core implies that~$\vphi$ is only with probability at most $\exp(-\Omega(pn^2))$ extendible to an $F$-free colouring of $Z\cup G(n,\eps p)$.
This allows us to use a union bound over all partition classes to get the desired upper bound 
on $\P(Z\cup G(n,\eps p)\nrightarrow (F)_2^e)$ of order~$o(1)$.

For the definition of the cores we will appeal to the hypergraph $\cH=\cH(Z,\Xi)$ which was defined in Section~\ref{subsec:mainlemmas}. Recall that $V(\cH)=e(Z)$ and hyperedges of~$\cH$ correspond 
to embeddings of $B$ in $K_n$, which are given by a carefully chosen subset 
$\Xi\subseteq \Psi_{B,n}^\prime$.
In fact, we shall select~$\Xi\subseteq \Psi_{B,n}^\prime$  
in such a way that we can  apply the structural result on independent sets of hypergraphs by 
Saxton and Thomason~\cite{ST12} to $\cH$ (see Lemma~\ref{lem:hauptlemma1}). In fact,
the cores then correspond to the complements of the almost independent sets from $\cJ$ given by the 
Saxton-Thomason theorem (Theorem~\ref{thm:ST}).  
This  yields a small family~$\cC$ of subsets of~$V(\cH)$, 
that means of size~$2^{o(pn^2)}$, such that the elements $C\in\cC$ are not too small and every hitting set of~$\cH$ contains at least one element from~$\cC$.

We then associate every $F$-free colouring $\vphi$ of~$Z$ with
a hitting set $A_{\vphi}^{\sigma}$ of~$\cH$ (for some $F$-free
colouring $\sigma$ of~$B$, see~part~\ref{it:hitting1} of Fact~\ref{fact:hittingset})
and thus we can associate to each such colouring~$\vphi$  a core $C\in\cC$ contained in~$A_{\vphi}^{\sigma}$.
This allows us to define the desired partition of the set of colourings~$\Phi$ using the ``small'' 
family of cores~$\cC$.
Finally, we use the union bound to estimate the probability that there is an $F$-free colouring of~$Z$ that can be extended to an $F$-free colouring of $Z\cup G(n,\eps p)$ by $o(1)$, which contradicts
$\P(Z\cup G(n,\eps p)\rightarrow (F)_2^e)<1-2\alpha$. Below we give the details of this proof.

\begin{proof}[Proof of Theorem~\ref{thm:Main}]
Let $F=F^\prime+\{a_1,a_2\}$ be a strictly balanced, nearly bipartite graph with~$F^\prime$ being bipartite
and assume for a contradiction that the property $G\rightarrow (F)_2^e$ does not have a sharp threshold.

We apply Corollary~\ref{cor:Boo} and obtain a function $p(n)=c(n)n^{-1/m_2(F)}$ with ${C_0<c(n)<C_1}$ for some $C_1> C_0>0$, constants $\frac{1}{3}>\alpha>0$, $\eps>0$ and a graph~$B$ with $B\nrightarrow (F)_2^e$.

For these parameters we apply Lemma~\ref{lem:hauptlemma1} and obtain 
constants $\alpha^\prime, \beta, \gamma>0$ and $L\in\NN$.
Set $\lambda=\beta/2$ and apply Lemma~\ref{lem:base}, which yields $\eta>0$.
Then let~$\cG_n$ be the family of graphs~$G$ on~$n$ vertices that satisfy the 
conclusions of Lemma~\ref{lem:hauptlemma1} and Lemma~\ref{lem:base} for the chosen parameters and $\frac{1}{4}pn^2\leq e(G) \leq pn^2$.
Since these properties hold a.a.s.\ in $G(n,p)$, it follows from
Corollary~\ref{cor:Boo}, that there are infinitely many $n\in\NN$ for which there is 
some~$Z\in\cG_n$ satisfying
\begin{enumerate}[label=\Rlabel]
\item\label{it:Boo1a} $\P(Z\cup h(B) \rightarrow (F)_2^e)>1-\alpha$, with $h\in \Psi_{B,n}$ chosen uniformly at random,
\item\label{it:Boo2b} $\P(Z\cup G(n,\eps p) \rightarrow (F)_2^e)<1-2\alpha$
\end{enumerate}
as well as by Lemma~\ref{lem:base}
\begin{enumerate}[label=\Tlabelb]
\item \label{prop T} $Z$ has the property $T(\lambda, \eta, F)$
\end{enumerate}
and
\begin{enumerate}[label=\Zlabelb]
\item \label{prop Z} $\frac{1}{4}pn^2\leq e(Z) \leq pn^2$.
\end{enumerate}

Owing to $Z\in \cG_n$ and~\ref{it:Boo1a} we can use Lemma~\ref{lem:hauptlemma1} to find some $\Xi_{B,n}\subseteq\Psi_{B,n}$ of size at least~$\alpha^\prime n^2$ with $Z\cup h(B)\rightarrow (F)_2^e$ for all $h\in \Xi_{B,n}$ such that the hypergraph $\cH=\cH(Z,\Xi_{B,n})$ is index consistent with a profile~$\pi$ of length $\l\leq L$ and such that there is a family~$\cC$ of subsets of~$V(\cH)$ with 
\begin{enumerate}[label=\Clabel]
\item\label{it:core1} $\log\vert \cC \vert \leq e(Z)^{1-\gamma}$,
\item\label{it:core2} $\vert C \vert \geq \beta e(Z)$ for all $C\in \cC$ and
\item\label{it:core3} every hitting set~$A$ of~$\cH$ contains a set $C\in \cC$.
\end{enumerate}

Our proof is by contradiction and 
we shall establish such a contradiction to the assertion~\ref{it:Boo2b}.

Let~$\Phi$ be the set of all $F$-free edge colourings of~$E(Z)$ and pick an arbitrary $F$-free colouring~$\sigma$ of~$B$.
We want to split~$\Phi$ into ``few'' classes.
For this we use the correspondence between any colouring $\vphi\in\Phi$ and the hitting set 
$A^{\sigma}_{\vphi}=A^{\sigma}_{\vphi}(Z,\Xi_{B,n})$ of~$\cH$ given by part~\ref{it:hitting1} 
of Fact~\ref{fact:hittingset}.
Moreover, for $C\in \cC$ we define 
\[
	\Phi_C=\{\vphi\in\Phi\colon C\subseteq A_\vphi^\sigma\}\,.
\]
Then $\Phi=\bigcup_{C\in\cC}\Phi_C$ (not necessarily disjoint) since by~\ref{it:core3} for every $\vphi\in\Phi$ 
the hitting set~$A_\vphi^\sigma$ contains some $C\in\cC$ and hence~$\vphi\in \Phi_{C}$.

Part~\ref{it:hitting2} of Fact~\ref{fact:hittingset}
asserts that $\vphi(z)=\vphi^\prime(z)$ for all $z\in A_\vphi^\sigma\cap A_{\vphi^\prime}^\sigma$ and colourings~$\vphi$, $\vphi^\prime\in\Phi$. In other words, all colourings in~$\Phi_C$ agree on~$C$ and, hence, there exists a monochromatic subset~$R_C\subseteq C$, say coloured red, of size at least~$\vert C\vert/2\geq \beta e(Z)/2=\lambda e(Z)$ (see~\ref{it:core2} and the choice of $\lambda$).
For the desired contradiction we add~$G(n,\eps p)$ to~$Z$. We have to show that
\[\P(Z\cup G(n,\eps p)\nrightarrow (F)_2^e)=o(1)\,.\]
For this purpose we find for all $F$-free colourings~$\vphi$ of~$Z$ an upper bound for the probability that~$\vphi$ is extendible to an $F$-free colouring of $Z\cup G(n,\eps p)$.
For $\vphi$ we use only the colouring on the associated core $C\subseteq A^\sigma_\vphi$, instead of the colouring on all edges of~$Z$.
In this way we can deal with all embeddings~$\vphi\in \Phi_C$ at once since they coincide on~$C$.

Since the red colour class $R_C$ contains at least $\lambda e(Z)$ edges it follows from
property~\ref{prop T}, that there are at least $\eta n^{v(F)}$ copies of~$F$ in the basegraph  
$\Base_{F}(R_C)$ of~$R_C$ w.r.t.\ $F$.
In an $F$-free colouring of $Z\cup G(n,\eps p)$ all edges in 
\[
	U_C=G(n,\eps p)\cap \Base_F(R_C)
\] 
have to be coloured blue since every edge in $\Base_F(R_C)$ completes a red copy of~$F^\prime$ in~$R_C$ to a copy of~$F$.
Consequently,~$\vphi$ cannot be extended to an $F$-free colouring of $Z\cup G(n,\eps p)$ if $U_C$ spans a copy of~$F$.
However, since $\Base_F(R_C)$ contains $\Omega(n^{v(F)})$ copies of~$F$ and $p=\Omega(n^{-1/m_2(F)})$ it follows from Janson's inequality~\cite{Ja90} (see also~\cite{JLR90}) that it is very unlikely that $U_C$ is $F$-free.
In fact, a standard application of Janson's inequality asserts that there exists some 
$\gamma^\prime=\gamma^\prime(\eps,\eta,C_0,C_1,F)$
such that 
\begin{equation}\label{eq:Janson}
	\P(F\nsubseteq G(n,\eps p)\cap \Base_F(R_C))=
	\P(F\nsubseteq U_C)\leq 
	\exp\left(-\gamma^\prime n^{2-\frac{1}{m_2(F)}}\right)\,.
\end{equation}
We then deduce the desired contradiction to~\ref{it:Boo2b} by
\begin{eqnarray*}
\P(Z\cup G(n,\eps p)\nrightarrow (F)_2^e)
&\leq&\vert \cC\vert\cdot  \max_{C\in\cC}\P(\exists \vphi\in\Phi_C\colon\vphi \text{ is extendible to }U_C)\\
&\overset{\text{\ref{it:core1}}}{\leq} &\exp\left(e(Z)^{1-\gamma}\right)\cdot \max_{C\in\cC}\P(F\nsubseteq U_C)\\
&\overset{\text{\ref{prop Z}}}{\leq}& \exp\left((pn^2)^{1-\gamma}\right)\cdot \max_{C\in\cC}\P(F\nsubseteq U_C)\\
&\overset{\text{\eqref{eq:Janson}}}{\leq}&
\exp\left((C_1n^{2-\frac{1}{m_2(F)}})^{1-\gamma}\right)\cdot\exp\left(-\gamma^\prime n^{2-\frac{1}{m_2(F)}}\right)\\
&<& \alpha\,,
\end{eqnarray*}
for sufficiently large $n$, since $\gamma>0$  and $C_1$, $\gamma$, and $\gamma'$ are constants 
independent of~$n$. This concludes the proof of 
Theorem~\ref{thm:Main}.
\end{proof}

\section{Proof of Lemma~\ref{lem:hauptlemma1}}\label{sec:proof Lemma 1}

The key tool to prove Lemma~\ref{lem:hauptlemma1} is the 
\emph{container theorem} (see Section~\ref{subsec:hypergraph containers}).
We shall apply Theorem~\ref{thm:ST} to the hypergraph  $\cH(Z,\Xi_{B,n})$. In order to satisfy the assumptions of  Theorem~\ref{thm:ST}
we may enforce some properties on the typical graph~$Z$ and the family of embeddings~$\Xi_{B,n}$.
Firstly in Section~\ref{subsec:Gnp}  we will formulate some properties on~$Z$ that hold a.a.s.\ for $G(n,p)$ and which will turn out to be useful 
for locating a suitable family of embeddings $\Xi_{B,n}\subseteq \Psi_{B,n}$ (see Section~\ref{subsec:Xi}).
In Section~\ref{subsec:Proofhauptlemma} we finally check that for those choices the assumptions of Theorem~\ref{thm:ST} are satisfied by 
the hypergraph $\cH(Z,\Xi_{B,n})$.

\subsection{Some typical properties of \texorpdfstring{$G(n,p)$}{G(n,p)}}\label{subsec:Gnp}
Corollary~\ref{cor:Boo} yields a family of embeddings of~$B$ into~$K_n$. We restrict ourselves to regular embeddings with foresight to the later parts of the proof.
Actually we want that for every edge $e\in E(Z)$ and every embedding $h$ there is at most one $b\in E(B)$ such that $e$ focuses on $h(b)$.
In addition there should be exactly one copy of~$F$ that contains~$e$ and~$h(b)$ if~$e$ focuses on~$h(b)$.
There are three ways such that this fails.

\begin{definition}\label{def:bad}
Let~$F$,~$B$,~$Z$ be graphs with $Z\subseteq K_n$. An embedding $h\in \Psi_{B,n}$ is \textit{bad} (with respect to~$F$
and~$Z$) if one of the following holds
\begin{enumerate}[label=\Blabel]
\item\label{it:bad1} either there is a copy~$F_1$ of~$F$ in $Z \cup h(B)$ that contains at least one edge of 
$E(Z)\setminus E(h(B))$ and at least two edges of~$E(h(B))$,
\item\label{it:bad2} or there are distinct copies~$F_1$ and~$F_2$ of~$F$ in $Z\cup h(B)$ and edges $e$, $f_1\neq f_2$ with $e\in E(Z)\setminus E(h(B))$ and $e\in E(F_1)\cap E(F_2)$, $f_1, f_2\in E(h(B))$ such that $f_1\in E(F_1)$ and $f_2\in E(F_2)$
\item\label{it:bad3} or there are distinct copies~$F_1$ and~$F_2$ of~$F$ in $Z\cup h(B)$ and edges~$e,f$ with $e\in E(Z)\setminus E(h(B))$ and $e\in E(F_1)\cap E(F_2)$, $f\in E(h(B))$ and $f\in E(F_1)\cap E(F_2)$.
\end{enumerate}
\end{definition}

Note that~\ref{it:bad3} would be a special case of~\ref{it:bad2} if we did not require~$f_1\neq f_2$ there.
However, for the later discussion it is better to distinguish these cases, 
and the idea of excluding embeddings $h$ because of~\ref{it:bad3} will be used in the proof 
of Lemma~\ref{lem:hauptlemma1} (see Lemma~\ref{lem:index}).

\begin{fact}\label{fact:bad imp reg}
For $F$, $B$ and $Z$ let $\Xi_{B,n}\subseteq \Psi_{B,n}$ be a family of embeddings such that properties~\ref{it:bad1} and~\ref{it:bad2} fail for every~$h\in \Xi_{B,n}$.
Then clearly the pair $(Z,\Xi_{B,n})$ is regular.
\end{fact}

We shall show that for the random graph $Z=G(n,p)$ only a few embeddings $h\in\Psi_{B,n}$ are bad (see~\ref{it:F--B} in Definition~\ref{def:Z} and Lemma~\ref{lem:G(BFnzg)} below), which enables us to focus on regular pairs 
$(Z,\Xi_{B,n})$. Moreover, we shall restrict to typical graphs~$Z$, which render a few more somewhat technical properties such as containing roughly the expected number of some special subgraphs. We discuss those properties below.

Let $\cF_{-}$ be the family of spanning subgraphs of $F$ obtained by removing 
some edge and for a graph $G$ we denote by $\cF_-(G)$ the copies of the members of $\cF_-$ in $G$.
Furthermore, for an edge $e\in E(G)$ let $\cF_{-}(G,e)$ be those copies in $\cF_-(G)$ that contain $e$.
For $e_1$, $e_2\in \binom{V(G)}{2}$ 
let $\cP(G,e_1,e_2)$ 
be the set of pairs $(F_1,F_2)$ of two edge disjoint subgraphs of~$G$
such that
\begin{itemize}
\item $F_1$ and~$F_2$ are copies of (possibly different) spanning subgraphs of~$F$, each of which obtained from~$F$ by removing two edges,
\item  the intersection $V(F_1)\cap V(F_2)=\{x_1,x_2,\dots, x_s\}$
contains at least two vertices, and
\item
$F_1 +\{x_1,x_2\} +e_1$ and $F_2 +\{x_1,x_2\} +e_2$ are isomorphic to~$F$.
\end{itemize}
For $s\geq 2$ let $\cP_{s}(G,e_1,e_2)\subseteq \cP(G,e_1,e_2)$ be the set of pairs as in $\cP(G,e_1,e_2)$ such that~$F_1$ and~$F_2$ intersect in exactly $s$ vertices.
Note that for $i=1,2$ by definition $e_i\neq \{x_1,x_2\}$ and~$e_i$
is not required to be an edge of~$G$.

These concepts lead to the following definition of ``good'' graphs $Z$, 
where we impose that the sizes of the introduced families defined above are close to the respective 
expectation in~$G(n,p)$.
Then Lemma~\ref{lem:G(BFnzg)}
states that a.a.s.\ $G(n,p)$ is indeed good for the right choice of parameters.

\begin{definition}\label{def:Z}
For graphs~$F$ and~$B$ and constants $D>0$, $\zeta>0$, $\delta>0$ and~$p\in(0,1)$ we consider the set of graphs $\cG_{B,F,n,p}(D,\zeta,\delta)$ on $n$ vertices that is given by $Z\in\cG_{B,F,n,p}(D,\zeta,\delta)$ if and only if 
\begin{enumerate}[label=\Zlabel]
\item\label{it:F--1} $\frac{1}{4}pn^2\leq e(Z) \leq pn^2$,
\item\label{it:F--2} $ |\cF_{-}(Z)|\leq D n^2$,
\item\label{it:F--3} $|\cF_{-}(Z,e)|\leq \frac{D}{p}$ for all $e\in E(Z)$,
\item\label{it:F--4} $|\cP(Z,e_1,e_2)|\leq \frac{D}{pn^\delta}$ for all but at most $\frac{Dpn^2}{n^\delta}$ pairs of distinct edges $e_1,e_2\in E(Z)$ and
\item\label{it:F--B} $\vert \lbrace h\in \Psi_{B,n}\colon h \text{ is bad w.r.t.\ }F\tand Z \rbrace \vert \leq \frac{\vert \Psi_{B,n}\vert}{n^\zeta}$.
\end{enumerate}
\end{definition}

The following Lemma shows that a.a.s.\ $G(n,p)\in \cG_{B,F,n,p}(D,\zeta,\delta)$ for~$D$ sufficiently large and~$\zeta$ and~$\delta$ sufficiently small (in fact, our choice of $\delta$ will imply $pn^\delta\rightarrow 0$).

\begin{lemma}\label{lem:G(BFnzg)}
For every strictly balanced graph~$F$, for every graph~$B$, and for all constants
$C_1\geq C_0>0$ there are constants $D>0$, $\zeta>0$, and $\delta$ 
with $0<\delta\leq \min\big\lbrace\tfrac{1}{m_2(F)},1-\tfrac{1}{m_2(F)}\big\rbrace$
such that for $C_0n^{-1/m_2(F)}\leq p\leq C_1n^{-1/m_2(F)}$ a.a.s.\ $G(n,p)\in\cG_{B,F,n,p}(D,\zeta,\delta)$.
\end{lemma}

We will split the proof into two parts: First we consider~\ref{it:F--1}-\ref{it:F--4} which deals with subgraphs of~$Z$ (Lemma~\ref{lem:F--inGnp}), and then we deal with the bad embeddings considered in~\ref{it:F--B} (Lemma~\ref{lem:G(BFnz)}).

\begin{lemma}\label{lem:F--inGnp}
For constants $C_1\geq C_0>0$, a strictly balanced graph~$F$, and~$p$ and~$n$
with $C_0 n^{-1/m_2(F)}\leq p\leq C_1n^{-1/m_2(F)}$ the following holds.
There exist constants $D>0$ and~$\delta$ with $0<\delta< \min\big\lbrace\tfrac{1}{m_2(F)},1-\tfrac{1}{m_2(F)}\big\rbrace$ such that a.a.s.\ $G(n,p)$ satisfies the properties~\ref{it:F--1},~\ref{it:F--2},~\ref{it:F--3}, and~\ref{it:F--4} with the parameters $p$, $D$, and $\delta$ and for the graph~$F$.
\end{lemma}

For the proof of Lemma~\ref{lem:F--inGnp} we note that property~\ref{it:F--1} follows directly from the concentration 
of the binomial distribution and~\ref{it:F--2} follows from~\ref{it:F--1} 
and~\ref{it:F--3}. The proof of~\ref{it:F--3} will make use of Spencer's extension lemma (Theorem~\ref{thm:Spencer}
stated below). Finally,~\ref{it:F--4} follows from a standard second moment argument.
Below we introduce the necessary notation for the statement of Theorem~\ref{thm:Spencer}.

For a graph~$H$ and an ordered proper subset $R=(x_1,\dots,x_r)$ of $V(H)$
the pair $(R,H)$ is called \textit{rooted graph} with \textit{roots}~$R$.
For an induced subgraph $H'=H[S]$ of $H$ with $\{x_1,\dots,x_r\}\subsetneq S$ we say 
$(R,H')$ is a rooted subgraph of $(R,H)$. We define the density of a rooted graph $(R,H)$ by 
\[
	\dens(R,H)=\frac{e(H)-e(H[R])}{v(H)-|R|}\,.
\]

Let $V(H)\setminus \{x_1,\dots,x_r\}=\{y_1,\dots,y_\nu\}$ for some $\nu\geq 1$.
For a graph~$G$  with some marked vertices $(x_1^\prime,\dots,x_r^\prime)$ an ordered tuple $(y_1^\prime,\dots,y_\nu^\prime)$ is called an \textit{$(R,H)$-extension of $(x_1^\prime,,\dots,x_r^\prime)$} if 
\begin{itemize}
\item the $y_i^\prime$ are distinct from each other and from the $x_j^\prime$,
\item $\{ x_i^\prime,y_j^\prime\}\in E(G) $ whenever $\{x_i,y_j\}\in E(H)$ and
\item $\{ y_i^\prime,y_j^\prime\}\in E(G)$ whenever $\{y_i,y_j\}\in E(H)$.
\end{itemize}
The number of $(R,H)$-extensions $(y_1^\prime,\dots,y_\nu^\prime)$ is denoted by $N(x_1^\prime,\dots,x_r^\prime)$.
Finally, we define $\mad(R,H)$ as the maximal average degree of a rooted graph $(R,H)$ by
\[
\mad(R,H)=\max\{\dens(R,H')\colon (R,H')\text{ is rooted subgraph of }(R,H)\}\,.
\]

\begin{theorem}[{\cite{Sp90}*{Theorem~3}}]\label{thm:Spencer}
Let $(R,H)$ be an arbitrary rooted graph and let $\eps>0$.
Then there exist $t$ such that if $p\geq n^{-1/\mad(R,H)}(\log n)^{1/t}$ then a.a.s.\ in~$G(n,p)$
\[(1-\eps)\E[N(\boldsymbol{x'})] <N(\boldsymbol{x'})<(1+\eps )\E[N(\boldsymbol{x'})]\]
for all~$\boldsymbol{x'}=(x'_1,\dots,x'_r)$ chosen from $[n]$.\qed
\end{theorem}

\begin{proof}[Proof of Lemma~\ref{lem:F--inGnp}]
\ref{it:F--1} This follows from an application of Chernoff's inequality.

\ref{it:F--2} As already mentioned this property follows from~\ref{it:F--1} 
and~\ref{it:F--3}. However, here is a standard direct proof based on the subgraph containment threshold in random graphs.
 
For $F_-\in\cF_{-}$ 
let~$X$ be the random variable that counts the number of copies of~$F_-$ contained in~$G(n,p)$.
Using that $p=\Theta(n^{-1/m_2(F)})$ combined with the balancedness of~$F$ yields
\[
\E[X]
= 
\Theta\left(n^{v(F)}p^{e(F)-1}\right)=\Theta(n^2)\,. 
\]
Moreover, by the definition of the $2$-density the expected number of copies of every non-trivial 
subgraph of $F_-\subset F$ is of order $\Omega(pn^2)$ and tends to infinity for $n\to\infty$.
Consequently,~$X$ converges to $\E[X]$ in probability (see, e.g.,~\cite{JLR00}*{Remark~3.7})
and we have $\P(X\geq 2\E[X])\rightarrow 0$ for $n\to\infty$.
Summing over all $F_-\in \cF_{-}$ yields the claim. 

\ref{it:F--3} Consider a graph $F_-\in \cF_-$ and remove some edge $\{x_1,x_2\}$ from~$F_-$ and call the resulting graph $F_{-2}$.  For $e\in \binom{[n]}{2}$ let~$X_e$ be the random variable that counts the number of copies of~$F_{-2}$ that build a copy of~$F_-$ by adding~$e$ and let~$X$ be the random variable that counts the number of copies of $F_{-2}$ contained in $G(n,p)$.

Now we can use Spencer's extension lemma (Theorem~\ref{thm:Spencer}).
We consider the rooted graph $((x_1,x_2),F_{-})$.
Let $\hat{F}$ be an induced  subgraph of $F_{-}$ such that $((x_1,x_2),\hat F)$ is a rooted subgraph of $((x_1,x_2),F_{-})$ 
which maximizes the density $\dens((x_1,x_2),\hat F)$.
Since the graph~$F\supsetneq  F_{-}\supseteq \hat F$ is strictly balanced we have
\[
m_2(F)>d_2(\hat{F})=\frac{e(\hat{F})-1}{v(\hat{F})-2}=\dens((x_1,x_2),\hat{F})=\mad((x_1,x_2),F_-)\,.
\]

Consequently, Theorem~\ref{thm:Spencer} applied with $\eps=1$ implies a.a.s.\
\[
	N(x'_1,x'_2)\leq 2\E(X_e)=O(p^{e(F)-2}n^{v(F)-2})
\]
for every $x'_1\neq x'_2\in[n]$. Owing to $p=\Theta(n^{-1/m_2(F)})$ and the (strict) balancedness of $F$
we have that $p^{e(F)}n^{v(F)}=\Theta(pn^2)$ and, consequently, for sufficiently large $D$
the claim follows by summing over all choices of $F_-\in\cF_{-}$ and $\{x_1,x_2\}\in E(F_-)$.

\ref{it:F--4} We show that this property holds a.a.s.\ for 
\begin{equation}\label{eq:14delta}
	\delta=\frac{1}{6} \min \Big\{\tfrac{1}{m_2(F)},1-\tfrac{1}{m_2(F)}\Big\} 
\end{equation}
and some $D>0$ independent of $n$. In the proof below we distinguish several cases.
In the first case we only look at configurations from $\cP_{2}(G(n,p),e_1,e_2)$.
Afterwards we consider configurations from $\cP_{s}(G(n,p),e_1,e_2)$ for~$s>2$.

\subsection*{Case 1: \texorpdfstring{$s=2$}{s=2}}
For two pairs $e_1\neq e_2\in \binom{[n]}{2}$ let $X_{e_1,e_2}$ be the random variable given by 
$|\cP_{2}(G(n,p),e_1,e_2)|$ and denote by~$v_1$ and~$u_1$ the elements of $e_1$ and by~$v_2$ and~$u_2$ the elements of~$e_2$.
We want to use Chebyshev's Inequality to obtain the claimed bound for most pairs.
Consequently, we estimate the expectation and variance of~$X_{e_1,e_2}$.
We distinguish between the cases $e_1\cap e_2 =\emptyset$ and $\vert e_1\cap e_2 \vert =1$. 

First let $e_1\cap e_2=\emptyset$. Since $C_0n^{-1/m_2(F)}\leq p\leq C_1n^{-1/m_2(F)}$ and~$F$ is strictly balanced we have $n^{v(F)}p^{e(F)}=\Theta(pn^2)$ and 
\begin{equation}
n^{v(F)-2} p^{e(F)-1}\leq C_1^{e(F)-1}\,.\label{eq:n^F-2 p^E-1=c}
\end{equation}
For $F_0\subseteq F$ with $v(F_0)\geq 2$ it follows from~$F$ being strictly balanced that there is some $d>0$ only depending on~$F$ and $C_0$ such that
\begin{equation}
n^{v(F_0)}p^{e(F_0)}\geq dpn^2\,.\label{eq:F_0 d}
\end{equation}
The expectation of~$X_{e_1,e_2}$ is

\begin{equation}\label{eq:exp}
\E[X_{e_1,e_2}]\leq e(F)^4 n^{2v(F)-6} p^{2e(F)-4}
\overset{\eqref{eq:n^F-2 p^E-1=c}}{\leq} e(F)^4 C_1^{2e(F)-2}n^{-2}p^{-2}
\end{equation}
and $\E[X_{e_1,e_2}]\rightarrow 0$ for $n$ tending to infinity since $p=\Theta(n^{-1/m_2(F)})$ and~$m_2(F)>1$.

Now we estimate the variance of $X_{e_1,e_2}$. We will show
\begin{equation*}
\Var (X_{e_1,e_2})\leq \frac{c}{n^2 p^2}\left(1+\frac{1}{np^2}\right)
\end{equation*}
for some constant~$c>0$ depending only on~$F$, $C_0$ and~$C_1$.
For this purpose let $(F_a,F_b)$ and $(F_c,F_d)$ be two different pairs of graphs that contribute to the number 
$|\cP_{2}(G(n,p),e_1,e_2)|$ with 
\[
	V(F_a)\cap V(F_b)=\{x_1,x_2\},\quad
	V(F_c)\cap V(F_d)=\{y_1,y_2\}
\]
and
\[
	V(F_a)\cap V(F_c)\supseteq \{u_1,v_1\} = e_1,\quad
	V(F_b)\cap V(F_d)\supseteq \{u_2,v_2\} = e_2\,.
\] 
Recall that $e_1$ and $e_2$ are by definition of
$\cP_{2}(G(n,p),e_1,e_2)$ not necessarily contained in $G(n,p)$ and they are 
not contained as edges in any of the subgraphs $F_a$, $F_b$, $F_c$, and $F_d$ (where
$s=2$ is used).
We denote by~$\cP_{e_1,e_2}^2$ the family of isomorphism types of possible pairs $((F_a,F_b),(F_c,F_d))$ such that the conditions above are satisfied. If it is clear from the context we will sometimes drop the subscripts~$e_1$ and $e_2$ to further ease the notation.

For $Q=((F_a,F_b),(F_c,F_d))\in \cP_{e_1,e_2}^2$ let $\cS_Q$ be the set of subsets of~$[n]$ of size $v(F_a\cup F_b\cup F_c\cup F_d)$ that contain $u_1$, $v_1$, $u_2$, and $v_2$. For $S\in\cS_Q$ let $1_S$ be the indicator random variable for the event ``there exists a copy of~$Q$ in~$G(n,p)$ on the vertex set~$S$''. Then
\begin{align}
\Var(X_{e_1,e_2})&\leq \E[X_{e_1,e_2}]+\sum_{Q\in\cP_{e_1,e_2}^2}\sum_{S\in \cS_Q} \P(1_S=1)\label{eq:Vare1e2}
\end{align}

For the estimation of the term $\sum_{Q\in\cP^2}\sum_{S\in \cS_Q} \P(1_S=1)$
we use the following notation.
For $\alpha, \beta \in \{a,b,c,d\}$ and $\square\in\{\cup,\cap\}$ we set 
\[
	v_{\alpha\square\beta}=v(F_\alpha\square F_\beta)\qand
	e_{\alpha\square\beta}=e(F_\alpha\square F_\beta)\,,
\]	 
where $F_\alpha\cap F_\beta$ and $F_\alpha\cup F_\beta$ denotes the normal union and intersection of two graphs.
Moreover, we can extend this to longer expressions of unions and intersections, like $v_{(\alpha\cap\beta)\cup\gamma}$,  and we will make use of this short hand notation in the calculations below.
We also set 
\begin{equation}\label{eq:absm}
	v_{\alpha\setminus \beta}=v_\alpha-v_{\alpha\cap \beta}\qqand
	e_{\alpha\setminus \beta}=e_\alpha-e_{\alpha\cap\beta}\,.
\end{equation}
Note that $e_{\alpha\setminus \beta}$ denotes the number of edges exclusively contained in $F_\alpha$, which 
does not necessarily coincide with $e(F_\alpha- V(F_\beta))$.
We estimate~$\sum_{Q\in\cP^2}\sum_{S\in \cS_Q} \P(1_S=1)$ by counting the number of choices for the vertices of the desired configuration and determine the number of needed edges.
Recalling that every $Q\in\cP_{e_1,e_2}^2$ corresponds to $((F_a,F_b),(F_c,F_d))$ we count those by first choosing~$(F_a,F_b)$, then~$F_c$ and then~$F_d$ and deal with the vertices and edges that are counted several times by looking at the intersections between the different copies of~$F$.

\begin{align}
&\sum_{Q\in\cP^2}\sum_{S\in \cS_Q} \P(1_S=1)\nonumber\\
&\qquad\overset{\phantom{\eqref{eq:absm}}}{\leq} \sum_{Q\in\cP^2} (4v(F)) ! \cdot n^{2v(F)-6}p^{2e(F)-4}\cdot n^{v_{c\setminus (a\cup b)}}p^{e_{c\setminus (a\cup b)}}\cdot n^{v_{d\setminus (a\cup b\cup c)}}p^{e_{d\setminus (a\cup b\cup c)}}\label{eq:VarFall1}\\
&\qquad\overset{\eqref{eq:absm}}{=}\!(4v(F)) ! \! \sum_{Q\in\cP^2} n^{4v(F)-6}p^{4e(F)-8}\cdot n^{-v_{c\cap (a\cup b)}}p^{-e_{c\cap (a\cup b)}}\cdot n^{-v_{d\cap (a\cup b\cup c)}}p^{-e_{d\cap (a\cup b\cup c)}}\nonumber\\
&\qquad\overset{\phantom{\eqref{eq:absm}}}{=}\!(4v(F)) !\!\sum_{Q\in\cP^2}\!\! n^2p^{-4}(n^{v(F)-2}p^{e(F)-1})^4 n^{-v_{c\cap (a\cup b)}}p^{-e_{c\cap (a\cup b)}} n^{-v_{d\cap (a\cup b\cup c)}}p^{-e_{d\cap (a\cup b\cup c)}}\nonumber\\
&\qquad\overset{\eqref{eq:n^F-2 p^E-1=c}}{\leq} C \sum_{Q\in\cP^2} n^2p^{-4}\cdot n^{-v_{c\cap (a\cup b)}}p^{-e_{c\cap (a\cup b)}}\cdot n^{-v_{d\cap (a\cup b\cup c)}}p^{-e_{d\cap (a\cup b\cup c)}}\,, \label{eq:Var0}
\end{align}
where $C>0$ is a constant depending only on~$F$ and~$C_1$.
For the estimation of
\begin{equation}\label{eq:fP}
f_Q(n,p):=n^2p^{-4}\cdot n^{-v_{c\cap (a\cup b)}}p^{-e_{c\cap (a\cup b)}}\cdot n^{-v_{(a\cup b\cup c)\cap d}}p^{-e_{(a\cup b\cup c)\cap d}}
\end{equation}
we distinguish several cases depending on the structure of~$Q$.

First we consider terms in~\eqref{eq:Var0} with~$\{x_1,x_2\}\subseteq V(F_c)$.
Since $\{x_1,x_2,v_1,u_1\}\subseteq V(F_a\cap F_c)$ and $F_a\cap F_c\subseteq F_a\subset F$ 
we know $F_0:=(F_a\cap F_c)+\{x_1,x_2\}+e_1\subseteq F$.
Therefore,
\[\frac{1}{n^{v_{a\cap c}}p^{e_{a\cap c}}}= \frac{p^2}{n^{v(F_0)}p^{e(F_0)}}\overset{\eqref{eq:F_0 d}}{\leq} \frac{p^2}{dpn^2}=\frac{p}{dn^{2}}\,.\]

Similarly, $(F_b\cap F_c)+\{x_1,x_2\}\subseteq F$ and $((F_a\cup F_b\cup F_c)\cap F_d)+\{y_1,y_2\}+e_2\subseteq F$. 
The same argument yields
 \[
\frac{1}{n^{v_{b\cap c}}p^{e_{b\cap c}}}\leq \frac{1}{dn^2} \quad\tand \quad
 \frac{1}{n^{v_{(a\cup b\cup c)\cap d}}p^{e_{(a\cup b\cup c)\cap d}}}\leq\frac{p}{dn^{2}}\,.
 \]

Applying these bounds and the facts that $v_{a\cap b\cap c}\leq 2$ and $e_{a\cap b\cap c}=0$
to~\eqref{eq:fP} yields

\begin{multline}
f_Q(n,p)
 = n^2p^{-4}\cdot n^{-v_{a\cap c}}p^{-e_{a\cap c}}\cdot n^{-v_{b\cap c}}p^{-e_{b\cap c}}\cdot n^{v_{a\cap b\cap c}}\cdot n^{-v_{(a\cup b\cup c)\cap d}}p^{-e_{(a\cup b\cup c)\cap d}}\\
\leq n^2p^{-4}\cdot\frac{p}{dn^2}\cdot \frac{1}{dn^2}\cdot n^2\cdot \frac{p}{dn^2}
=\frac{1}{d^3p^2n^2}\label{eq:fp1} \,.
\end{multline}
By symmetry we obtain the same estimate in the case that~$\{x_1,x_2\}\subseteq V(F_d)$
and in the remaining case we may assume
\begin{enumerate}[label=\Ilabelb]
\item \label{eq:Schnitt1} $\vert V(F_c)\cap \{x_1,x_2\}\vert\leq 1$ and $\vert V(F_d)\cap \{x_1,x_2\}\vert\leq 1$.
\end{enumerate}

Next we consider those terms in~\eqref{eq:Var0} with~\ref{eq:Schnitt1} and $v_{b\cap c}\geq 2$.
By~\ref{eq:Schnitt1} we have $v_{a\cap b\cap c}\leq 1$.
We proceed in a similar way as above.
This time we use that $(F_a\cap F_c)+e_1\subseteq F$ and similarly that $((F_a\cup F_b\cup F_c)\cap F_d)+e_2+\{y_1,y_2\}\subseteq F$ and, therefore, 
\[
\frac{1}{n^{v_{a\cap c}}p^{e_{a\cap c}}}\overset{\eqref{eq:F_0 d}}{\leq}\frac{1}{dn^2}
\qqand 
\frac{1}{n^{v_{(a\cup b\cup c)\cap d}}p^{e_{(a\cup b\cup c)\cap d}}}\overset{\eqref{eq:F_0 d}}{\leq} \frac{p}{dn^2}\,.
\]
Moreover, since we assume $v_{b\cap c}\geq 2$ we can apply \eqref{eq:F_0 d} with
$F_0=F_b\cap F_c$
\[
	\frac{1}{n^{v_{b\cap c}}p^{e_{b\cap c}}}\leq \frac{1}{dpn^2}\,.
\]
Combining these bounds with~\eqref{eq:fP} and $v_{a\cap b\cap c}\leq 1$ and $e_{a\cap b\cap c}=0$ yields
\begin{multline}
f_Q(n,p)
\leq n^2p^{-4}\cdot n^{-v_{a\cap c}}p^{-e_{a\cap c}}\cdot n^{-v_{b\cap c}}p^{-e_{b\cap c}}\cdot n\cdot n^{-v_{(a\cup b\cup c)\cap d}}p^{-e_{(a\cup b\cup c)\cap d}}\\
\leq n^2p^{-4}\cdot \frac{1}{dn^2}\cdot \frac{1}{dpn^2}\cdot n\cdot \frac{p}{dn^2}
=\frac{1}{d^3p^4n^3}\label{eq:fp2}\,.
\end{multline}

Next we consider the subcase of~\ref{eq:Schnitt1} when 
\[
	v_{b\cap c} =1 \qqand V(F_c)\cap\{x_1,x_2\}=\emptyset\,.
\]
Then we have $e_{b\cap c}=0$ and $v_{a\cap b\cap c}=0$.
Since $(F_a\cap F_c)+e_1\subseteq F$ and $((F_a\cup F_b\cup F_c)\cap F_d)+e_2+\{y_1,y_2\}\subseteq F$ we get
\[\frac{1}{n^{v_{a\cap c}}p^{e_{a\cap c}}}\overset{\eqref{eq:F_0 d}}{\leq}\frac{1}{dn^2}
\qqand \frac{1}{n^{v_{(a\cup b\cup c)\cap d}}p^{e_{(a\cup b\cup c)\cap d}}}\overset{\eqref{eq:F_0 d}}{\leq} \frac{p}{dn^2}\,.\]

Consequently, in this case we have
\begin{align}
f_Q(n,p)
&= n^2p^{-4}\cdot n^{-v_{a\cap c}-v_{b\cap c}+v_{a\cap b\cap c}-v_{(a\cup b\cup c)\cap d}}p^{-e_{a\cap c}-e_{b\cap c}+e_{a\cap b\cap c}-e_{(a\cup b\cup c)\cap d}}\nonumber\\
&\leq n^2p^{-4}\cdot n^{-v_{a\cap c}}p^{-e_{a\cap c}}\cdot n^{-1}\cdot n^{-v_{(a\cup b\cup c)\cap d}}p^{-e_{(a\cup b\cup c)\cap d}}\nonumber\\
&\leq n^2p^{-4}\cdot \frac{1}{dn^2}\cdot n^{-1}\cdot \frac{p}{dn^2}
=\frac{1}{d^2p^3n^3}\label{eq:fp3}\,.
\end{align}

For the last remaining cases we consider summands in~\eqref{eq:Var0} with~\ref{eq:Schnitt1} and 
\begin{enumerate}[label=\Alabel]
\item either $v_{b\cap c}=1$ and $V(F_c)\cap \{x_1,x_2\}\neq \emptyset$ 
(and, hence,  $V(F_b)\cap V(F_c)\subsetneq\{x_1,x_2\}$),
\item or $v_{b\cap c}=0$.
\end{enumerate}
In both cases together with~\ref{eq:Schnitt1} we get
\begin{equation}\label{eq:bdac}
v_{b\cap (a\cup c)\cap d}=|\{x_1,x_2\}\cap V(F_d)| \leq 1\,.
\end{equation}
Based on~\eqref{eq:bdac} we treat both subcases in same way.
We consider $((F_a\cup F_b)\cap F_c)+e_1\subseteq F$, $(F_b\cap F_d)+e_2\subseteq F$ and $((F_a\cup F_c)\cap F_d)+\{y_1,y_2\}\subseteq F$ and get
\[
\frac{1}{n^{v_{(a\cup b)\cap c}}p^{e_{(a\cup b)\cap c}}}\overset{\eqref{eq:F_0 d}}{\leq}\frac{1}{dn^2}\, , \,
\frac{1}{n^{v_{b\cap d}}p^{e_{b\cap d}}}\overset{\eqref{eq:F_0 d}}{\leq} \frac{1}{dn^2}\qqand
\frac{1}{n^{v_{(a\cup c)\cap d}}p^{e_{(a\cup c)\cap d}}}\overset{\eqref{eq:F_0 d}}{\leq} \frac{1}{dn^2}\,,
\]
which leads to
\begin{align}
f_Q(n,p)&\overset{\phantom{\eqref{eq:bdac}}}{=} n^2p^{-4}\cdot n^{-v_{(a\cup b)\cap c}-v_{b\cap d}-v_{(a\cup c)\cap d}+v_{b\cap (a\cup c)\cap d}}\cdot p^{-e_{(a\cup b)\cap c}-e_{b\cap d}-e_{(a\cup c)\cap d}+e_{b\cap (a\cup c)\cap d}}\nonumber\\
&\overset{\eqref{eq:bdac}}{\leq} n^2p^{-4}\cdot n^{-v_{(a\cup b)\cap c}}p^{-e_{(a\cup b)\cap c}}\cdot n^{-v_{b\cap d}}p^{-e_{b\cap d}}\cdot n^{-v_{(a\cup c)\cap d}}p^{-e_{(a\cup c)\cap d}}\cdot n\nonumber\\
&\overset{\phantom{\eqref{eq:bdac}}}{\leq} n^2p^{-4}\cdot \left(\frac{1}{dn^2}\right)^3\cdot n
=\frac{1}{d^3p^4n^3}\label{eq:fp4}\,.
\end{align}

Using the bounds from~\eqref{eq:fp1}, \eqref{eq:fp2}, \eqref{eq:fp3} and~\eqref{eq:fp4} and $pn\rightarrow\infty$ for $n\rightarrow\infty$ we summarize that there are constants $c^\prime,c>0$ only depending on $F,C_0$ and~$C_1$ such that for sufficiently large~$n$
\[
f_Q(n,p)\leq c^\prime\left(\frac{1}{p^2n^2}+\frac{1}{p^4n^3}\right)\,.
\]
Since the sum in~\eqref{eq:Var0} has finitely many summands, together with~\eqref{eq:Vare1e2} and~\eqref{eq:Var0} it follows that
\begin{equation}
\Var(X_{e_1,e_2})\leq \frac{c}{p^2n^2}\left(1+\frac{1}{p^2n}\right). \label{eq:Var2}
\end{equation}

Recall that we want to show that there are at most $Dpn^2n^{-\delta}$ pairs of edges $e_1,e_2$ in~$G(n,p)$ so that $X_{e_1,e_2}>Dp^{-1}n^{-\delta}$ for some constant~$D>0$ independent of~$n$ and $\delta>0$ chosen 
in~\eqref{eq:14delta}.
For this purpose we use Markov's Inequality and Chebyshev's Inequality.
Let $t=p^{-1}n^{-\delta}$, then 
Chebyshev's Inequality tells us
\[\P(X_{e_1,e_2}\geq \E[X_{e_1,e_2}]+t)\leq \frac{\Var(X_{e_1,e_2})}{t^2}\,.\]

Let~$X$ be the number of pairs $(e_1,e_2)\in\binom{E(Z)}{2}$ with $X_{e_1,e_2}\geq 2p^{-1}n^{-\delta}$ and~$e_1\cap e_2=\emptyset$.
Since $\E[X_{e_1,e_2}]\leq t$ we have
\begin{align}
\E[X]&\leq \binom{pn^2}{2}\P(X_{e_1,e_2}\geq \E[X_{e_1,e_2}]+t) \label{eq:Var3}\\
&\leq \frac{p^2n^4}{2}\cdot \frac{cp^2 n^{2\delta}}{p^2n^2}\left(1+\frac{1}{np^2}\right) 
=\frac{1}{2}cp^2n^{2+2\delta}\left(1+\frac{1}{np^2}\right)\label{eq:Var4}.
\end{align}

We distinguish the cases $n^{-1}p^{-2}> 1$ and $n^{-1}p^{-2}\leq1$.
For $n^{-1}p^{-2}> 1$ we have for sufficiently large~$n$
\[\E[X]\leq \frac{cp^2n^{2+2\delta}}{np^2} \leq cn^{1+2\delta}\leq pn^{2-2\delta},\]
where the last inequality follows from our choice of~$\delta<\frac{1}{4}(1-\frac{1}{m_2(F)})$.

For the case $n^{-1}p^{-2}\leq 1$ we have for sufficiently large~$n$
\[\E[X]\leq cp^2n^{2+2\delta}\leq pn^{2-2\delta}\]
where the last inequality follows by the choice of $\delta<\frac{1}{4m_2(F)}$.
Consequently, $\E[X]\leq pn^{2-2\delta}$ and by Markov's Inequality 
\[\P(X>pn^{2-\delta})\leq\frac{\E[X]}{pn^{2-\delta}}\leq n^{-\delta}\]
thus a.a.s.\ $X\leq pn^{2-\delta}$.
For sufficiently large~$n$ this finishes the case $e_1\cap e_2 =\emptyset$.

It remains the case when $\vert e_1\cap e_2\vert=1$.
Now let $e_1,e_2\in\binom{[n]}{2}$ with $\vert e_1\cap e_2\vert =1$.
We repeat essentially the same calculations of the first case $e_1\cap e_2=\emptyset$ with the following differences.

\begin{itemize}
\item For the expectation of~$X_{e_1,e_2}$ in~\eqref{eq:exp} we get
\[
	\E[X_{e_1,e_2}]=O\left(\frac{1}{np^2}\right)\,.
\]
\item For the variance we will show
\[
	\Var(X_{e_1,e_2})\leq \frac{c}{np^2}\left(1+\frac{1}{np^2}\right)\,.
\]
In the calculation of the variance there is essentially one difference compared to the case $e_1\cap e_2=\emptyset$.
In~\eqref{eq:VarFall1} we get
\[
	v_{a\cup b}-\vert\{x_1,x_2\}\cup \{v_1,u_1\}\cup \{v_2,u_2\}\vert \leq 2v(F)-5
\]
instead of $2v(F)-6$ which leads to an additional $n$ factor.
This $n$ factor carries over to
\begin{equation}
f_Q(n,p):=n^3p^{-4}\cdot n^{-v_{c\cap (a\cup b)}}p^{-e_{c\cap (a\cup b)}}\cdot n^{-v_{(a\cup b\cup c)\cap d}}p^{-e_{(a\cup b\cup c)\cap d}}
\end{equation}
in~\eqref{eq:fP}.

For the following case distinction we repeat in the case $\{x_1,x_2\}\subseteq V(F_c)$ the calculation, but keep the additional $n$ factor.
Consequently we get in~\eqref{eq:fp1}
\[
	f_Q(n,p)=O\left(\frac{1}{p^2n}\right) \,.
\]
Similarly we get with the additional $n$ factor in~\eqref{eq:fp2}
\[
	f_Q(n,p)=O\left(\frac{1}{p^4n^2}\right)\,.
\]

The case $v_{b\cap c}=1$ and $V(F_c)\cap \{x_1,x_2\}=\emptyset$ disappears since $F_b$ and $F_c$ intersect at least in $e_1\cap e_2\subseteq\{x_1,x_2\}$.
For the same reason the case $v_{b\cap c}=0$ disappears.
For the last remaining case in~\eqref{eq:fp4} we get again the same bound with an additional factor of $n$
\[
	f_Q(n,p)=O\left(\frac{1}{p^4n^2}\right)\,.
\]
Consequently
\[
	\Var(X_{e_1,e_2})\leq \frac{c}{np^2}\left(1+\frac{1}{np^2}\right)\,.
\]
\item The expectation still satisfies $\E[X_{e_1,e_2}]\leq t$ for the same choice of $t=p^{-1}n^{-\delta}$.
This follows since $\E[X_{e_1,e_2}]=O(\frac{1}{np^2})$, $t=\frac{1}{pn^\delta}$ and $\delta<1-\frac{1}{m_2(F)}$.
\item Let~$X^\prime$ be the number of pairs $(e_1,e_2)\in\binom{E(Z)}{2}$ satisfying $X_{e_1,e_2}\geq 2p^{-1}n^{-\delta}$ and $\vert e_1\cap e_2 \vert =1$.
We know by the condition~$\vert e_1\cap e_2 \vert =1$ that $X^\prime\leq 2p^2n^3$, thus we get with $X^\prime$ instead of $X$ in~\eqref{eq:Var3} a factor of $2p^2n^3$ instead of $\binom{pn^2}{2}$ which results in a factor of~$n^{-1}$ compared to the first case.
Consequently the $n^{-1}$ factor cancels with the $n$ factor above which leads to the same order of magnitude in~\eqref{eq:Var4}.
Then the rest of the proof is the same as in the first case.
\end{itemize}

Setting~$D^\prime\geq 2$ sufficiently large such that $2 p^{-1}n^{-\delta}\leq \frac{D^\prime pn^2}{n^\delta}$ then yields 
\begin{equation}\label{eq:case1}
	|\cP_{2}(Z,e_1,e_2)|\leq \frac{D^\prime}{pn^\delta}
\end{equation}
for all but at most $\frac{D^\prime pn^2}{n^\delta}$ pairs of edges $e_1,e_2\in E(Z)$.

\subsection*{Case 2: \texorpdfstring{$s>2$}{s>2}}
We consider configurations from $\cP_{s}(G(n,p),e_1,e_2)$ with $s>2$.
For two pairs $e_1\neq e_2\in \binom{[n]}{2}$ let $Y_{e_1,e_2}$ be the random variable given by 
$|\cP_{s}(G(n,p),e_1,e_2)|$.
Here it is sufficient to use Markov's inequality instead of Chebyshev's inequality which will allow us to avoid the calculation of the variance, but we still have to distinguish the cases $e_1\cap e_2=\emptyset$ and $\vert e_1\cap e_2\vert =1$.

For the first case let $e_1\cap e_2=\emptyset$.
The expectation of $Y_{e_1,e_2}$ is
\[
\E[Y_{e_1,e_2}]\leq e(F)^4 n^{2v(F)-4-s}v(F)^s p^{2e(F)-4}
\overset{\eqref{eq:n^F-2 p^E-1=c}}{\leq} e(F)^4v(F)^s C_1^{2e(F)-2}n^{-s}p^{-2}
\leq C^\prime n^{-3}p^{-2}
\]
with $C^\prime=e(F)^4v(F)^s C_1^{2e(F)-2}$.
We use Markov's inequality and get
\[
	\P\left(Y_{e_1,e_2}\geq \frac{1}{pn^\delta}\right)\leq C^\prime n^{-3}p^{-2}\cdot pn^\delta=C^\prime p^{-1}n^{-3+\delta}\,.
\]
Let $Y$ be the number of pairs $e_1,e_2\in E(Z)$ with $e_1\cap e_2=\emptyset$ and $Y_{e_1,e_2}\geq p^{-1}n^{-\delta}$.
Then
\[
	\E[Y]\leq \binom{pn^2}{2}C^\prime n^{-3+\delta}p^{-1}\leq \frac{C^\prime pn^{1+\delta}}{2} 
\]
and a second use of Markov's inequality yields
\[
	\P(Y\geq pn^{2-\delta})\leq \frac{C^\prime pn^{1+\delta}}{2 pn^{2-\delta}}=o(1)
\]
where the last inequality follows from our choice $\delta< 1/2$ and for sufficiently large~$n$.

We repeat the same proof for the case $\vert e_1\cap e_2\vert =1$ with the following differences.
\begin{itemize}
\item $\E[Y_{e_1,e_2}]\leq C^{\prime\prime}n^{-2} p^{-2}$ for some $C^{\prime\prime}>0$.
\item $\P\left(Y_{e_1,e_2}\geq \frac{1}{pn^\delta}\right)\leq C^{\prime\prime} p^{-1}n^{-2+\delta}$.
\item $\E[Y]\leq 2p^2n^3 C^{\prime\prime}p^{-1} n^{-2+\delta}\leq 2C^{\prime\prime} pn^{1+\delta}$.
\item $	\P(Y\geq pn^{2-\delta})\leq \frac{2C^{\prime\prime} pn^{1+\delta}}{pn^{2-\delta}}=o(1)$.
\end{itemize}

Consequently for all $s\geq 3$ we have $\vert \cP_{s}(G(n,p),e_1,e_2)\vert \leq p^{-1}n^{-\delta}$ for all but at most~$pn^{2-\delta}$ pairs of edges $e_1,e_2\in E(Z)$.
Together with~\eqref{eq:case1} this concludes the proof of \ref{it:F--4} and finishes the proof of Lemma~\ref{lem:F--inGnp}.
\end{proof}

The next lemma concerns property~\ref{it:F--B}, which bounds the number of bad embeddings as defined in  Definition~\ref{def:bad}.

\begin{lemma}\label{lem:G(BFnz)}
For all graphs~$B$ and all strictly balanced graphs~$F$, for all $C_1\geq C_0>0$ and for $C_0n^{-1/m_2(F)}\leq p\leq C_1n^{-1/m_2(F)}$ there exists $\zeta>0$ such that a.a.s.\  $G(n,p)$ satisfies~\ref{it:F--B}.
\end{lemma}

\begin{proof}[Proof of Lemma~\ref{lem:G(BFnz)}]
We shall show that there exist a~$\xi>0$ such that for any given $h\in\Psi_{B,n}$ we have for sufficiently large~$n$
\[\P(h \text{ is bad w.r.t.\ } F\tand G(n,p))\leq n^{-\xi}\,.\]
 Then the lemma follows from Markov's inequality with~$\zeta=\xi/2$.

Let $h\in\Psi_{B,n}$ be fixed.
We first consider the case that~$h$ is bad w.r.t.\ $F$ and $G(n,p)$ because of~\ref{it:bad1}.
Since~$F$ is strictly balanced, for all proper subgraphs $F_0\subsetneq F$ with $e(F_0)\geq 2$ we have
\begin{align}
p^{e(F_0)}n^{v(F_0)} = pn^2 \cdot p^{e(F_0)-1}n^{v(F_0)-2}
 &\geq  pn^2\cdot C_0^{e(F_0)-1} n^{-\frac{1}{m_2(F)}(e(F_0)-1)+v(F_0)-2}\nonumber\\
&= pn^2\cdot C_0^{e(F_0)-1} n^{(e(F_0)-1)\big(\frac{v(F_0)-2}{e(F_0)-1}-\frac{1}{m_2(F)}\big)}\nonumber\\
&= pn^2\cdot C_0^{e(F_0)-1} n^{(e(F_0)-1)\big(\frac{1}{d_2(F_0)}-\frac{1}{d_2(F)}\big)}\geq pn^2\cdot n^{\xi^\prime}\label{eq:ex:F_0}
\end{align}
for some $\xi^\prime>0$.
We bound the probability for $h$ being bad because of case~\ref{it:bad1} by estimating the number of configurations leading to this event.
In this case $F_0$ stands for the part of $F$ that is contained in  $h(B)$ and hence consists of at least two edges. Using again $n^{v(F)-2}p^{e(F)-1}\leq C_1^{e(F)-1}$ yields
\begin{align*}
\P(h\text{ is bad by~\ref{it:bad1}})
&\overset{\phantom{\eqref{eq:ex:F_0}}}{\leq} \sum_{F_0\subsetneq F, e(F_0)\geq 2} v(B)^{v(F_0)} n^{v(F)-v(F_0)}p^{e(F)-e(F_0)}\\
&\overset{\eqref{eq:ex:F_0}}{\leq} \sum_{F_0\subsetneq F, e(F_0)\geq 2} v(B)^{v(F_0)} C_1^{e(F)-1} n^{-\xi^\prime}
\leq n^{-\xi_1}
\end{align*}
for some $\xi_1>0$ and sufficiently large~$n$.

When we address the case~\ref{it:bad2} we can assume that~$h$ is not bad because of case~\ref{it:bad1}.
Hence, it suffices to consider copies $F_1$ and $F_2$ of $F$ each intersecting $h(B)$ in precisely one edge 
and $F_0:=F_1\cap F_2$ having no edge in~$h(B)$.
Again we will use $n^{v(F)-2}p^{e(F)-1}\leq C_1^{e(F)-1}$ and that $n^{v(F_0)}p^{e(F_0)}\geq dpn^2$ for $F_0\subsetneq F$ with $e(F_0)\geq 1$ for some $d>0$ only depending on~$F$ and~$C_0$ (see~\eqref{eq:F_0 d}).
Note that two fixed edges of $h(B)$ determine at least three vertices of $F_1\cup F_2$. 
\begin{align*}
\P(h \text{ is bad by~\ref{it:bad2} and not by~\ref{it:bad1}})
&\leq \sum_{\substack{F_0\subsetneq F\\ e(F_0)\geq 1}} v(B)^4 n^{2v(F)-v(F_0)-3}p^{2e(F)-e(F_0)-2}\\
&\leq \sum_{F_0} v(B)^4 C_1^{2e(F)-2}\frac{n}{p^{e(F_0)}n^{v(F_0)}}\\
&\leq \sum_{F_0} v(B)^4 C_1^{2e(F)-2}\frac{1}{dpn}\\
&\leq \sum_{F_0} v(B)^4 C_1^{2e(F)-3}d^{-1}n^{-(1-\frac{1}{m_2(F)})}
\leq n^{-\xi_2}
\end{align*}
for some $\xi_2>0$ since $m_2(F)>1$.

For case~\ref{it:bad3} we assume that~$h$ is not bad because of case~\ref{it:bad1} or case~\ref{it:bad2}.
Again we bound the probability by the expected number of options to obtain a configuration as in~\ref{it:bad3}. 
In this case $F_0$ stands for the intersection of two different copies of~$F$ and includes at least two edges, $e$ and $f$ from~\ref{it:bad3}, where $f$ is also contained in $h(B)$.

\begin{align*}
\P(h \text{ is bad by~\ref{it:bad3} and not by~\ref{it:bad1} or~\ref{it:bad2}})
&\overset{\phantom{\eqref{eq:ex:F_0}}}{\leq} \sum_{\substack{F_0\subsetneq F\\ e(F_0)\geq 2}} v(B)^2 n^{2v(F)-v(F_0)-2}p^{2e(F)-e(F_0)-1}\\
&\overset{\phantom{\eqref{eq:ex:F_0}}}{\leq} \sum_{F_0} v(B)^2 C_1^{2e(F)-2} \cdot pn^2\cdot \frac{1}{p^{e(F_0)}n^{v(F_0)}} \\
&\overset{\eqref{eq:ex:F_0}}{\leq} n^{-\xi_3}
\end{align*}
for some $\xi_3>0$ and, hence, 
 $\P(h\text{ is bad})\leq n^{-\xi}$ for any $0<\xi<\min\{\xi_1,\xi_2,\xi_3\}$ and sufficiently large~$n$.
\end{proof}

\subsection{Restricting embeddings of~\texorpdfstring{$B$}{B}}\label{subsec:Xi}
In this section we focus on restricting the family $\Psi_{B,n}$ of all embeddings~$B$ in~$K_n$ to a suitable
subset~$\Xi_{B,n}$  so that we can apply Theorem~\ref{thm:ST} for the proof of Lemma~\ref{lem:hauptlemma1}.
In particular, our choice of~$\Xi_{B,n}$  will ensure conditions on the maximum degree and maximum pair degree 
of~$\cH=\cH(Z,\Xi_{B,n})$. For the control of the pair degree of~$\cH$ the following definition will be useful.

\begin{definition}\label{def:connection}
For a pair of edges $e_1,e_2\in E(Z)$ and an embedding $h\in\Xi_{B,n}\subseteq \Psi_{B,n}$ we write $e_1\approx_h e_2$ if~$e_1$ and~$e_2$ both focus on~$h(B)$.
Moreover, if~$e_1$ and~$e_2$ focus jointly on only one edge of~$h(B)$, then we write $e_1\sim_h e_2$.
We denote by $c_{\Xi_{B,n}}(e_1,e_2)$ the number of~$h\in\Xi_{B,n}$ such that~$e_1\approx_h e_2$.
\end{definition}

In the next definition and lemma we define the properties of the desired family of embeddings.

\begin{definition}\label{def:normal}
Let~$F$,~$B$ be graphs and let $\alpha>0$.
We call a family $\Xi_{B,n}\subseteq \Psi_{B,n}$ of embeddings of~$B$ into~$K_n$ $\alpha$-\textit{normal} if the following conditions are satisfied.
\begin{enumerate}[label=\Nlabel]
\item\label{it:normal1} $\vert \Xi_{B,n}\vert \geq \alpha n^2$ and
\item\label{it:normal2} $\vert V(h(B))\cap V(h^\prime(B))\vert \leq 1$ for all $h\neq h^\prime\in \Xi_{B,n}$.
\end{enumerate} 
\end{definition}

\begin{lemma}\label{lem:disEmb} Let~$F$ and~$B$ be graphs.
For all constants $\frac{1}{3}>\alpha>0$, $D>0$, $1>\zeta>0$, $\min \{\frac{1}{m_2(F)},1-\frac{1}{m_2(F)}\}>\delta>0$, and $C_1>C_0>0$ there exists $n_0\in\NN$ such that for all $n\geq n_0$ and  $C_0n^{-1/m_2(F)}\leq p\leq C_1n^{-1/m_2(F)}$ the following holds.
If $Z\in \cG_{B,F,n,p}(D,\zeta,\delta)$ and \[\P(Z\cup h(B)\rightarrow (F)_2^e)> 1-\alpha\] where $h\in \Psi_{B,n}$ chosen uniformly at random then there exists $\Xi_{B,n}^0\subseteq \Psi_{B,n}$ such that 
\begin{enumerate}[label=\Xilabel]
\item\label{it:disem1} $\Xi_{B,n}^0$ is $\talpha$-normal for $\talpha=\talpha (B)=\frac{1}{13v(B)^4v(B)!}>0$,
\item\label{it:disem2} $Z\cup h(B) \rightarrow (F)_2^e$ for all $h\in \Xi_{B,n}^0$,
\item\label{it:disem3} for all pairs $\{e_1,e_2\}\in \binom{E(Z)}{2}$ we have $c_{\Xi_{B,n}^0}(e_1,e_2)\leq \frac{1}{pn^{\delta/2}}$,
\item\label{it:disem4} $h$ is not bad w.r.t.\ $F$ and $Z$ for all $h\in\Xi_{B,n}^0$ (see Definition~\ref{def:bad}), and
\item\label{it:disem5} for all $h\in\Xi_{B,n}^0$ we have $E(h(B))\cap E(Z)=\emptyset$.
\end{enumerate}  
A family~$\Xi_{B,n}^0$ is $(\talpha,Z)$-normal if it satisfies conditions~\ref{it:disem1},~\ref{it:disem2},~\ref{it:disem3},~\ref{it:disem4}, and~\ref{it:disem5} for a given $Z\in \cG_{B,F,n,p}(D,\zeta,\delta)$.
\end{lemma}

\begin{proof}[Proof of Lemma~\ref{lem:disEmb}]
Given~$F$, $B$ and the constants as above we set 
\[\talpha=\frac{1}{13v(B)^4v(B)!}\,.\]
Let $Z\in \cG_{B,F,n,p}(D,\zeta,\delta)$ and suppose $\P(Z\cup h(B)\rightarrow (F)_2^e)>1-\alpha$.

For the construction of~$\Xi_{B,n}^0$ we start with the family~$\Psi_{B,n}$ and remove embeddings that do not satisfy property~\ref{it:disem2}, embeddings that do not satisfy property~\ref{it:disem4} and embeddings that will later lead to problems for~\ref{it:disem3}.
After that we choose at random~$2\talpha n^2$ embeddings which will induce property~\ref{it:disem3} and show that after deleting the embeddings that intersect in more than one vertex we keep $C\talpha n^2$ of them with $C>1$.
Afterwards we remove embeddings not satisfying~\ref{it:disem5}.
Since $e(Z)=\Theta(pn^2)$ we keep at least $(C\talpha-o(1)) n^2>\talpha n^2$ embeddings $h$,
which finishes the proof. 

Since $\P(Z\cup h(B)\rightarrow (F)_2^e)>1-\alpha>2/3$ there is a family $\Psi_{B,n}^1\subseteq\Psi_{B,n}$ of embeddings of~$B$ of size $\frac{2}{3}\vert \Psi_{B,n}\vert $ such that $Z\cup h(B)\rightarrow (F)_2^e$ for all $h\in\Psi_{B,n}^1$, i.e., $\Psi_{B,n}^1$ satisfies~\ref{it:disem2}.

Moreover, since $Z\in\cG_{B,F,n,p}(D,\zeta,\delta)$ there are at most $n^{-\zeta}\vert \Psi_{B,n}\vert $ embeddings that are bad~w.r.t.\ $F$ and~$Z$.
We remove those bad embeddings from $\Psi_{B,n}^1$.
In this way for sufficiently large~$n$ we obtain a family $\Psi_{B,n}^2\subseteq\Psi_{B,n}^1$ of size at least $\frac{1}{2}\vert\Psi_{B,n}\vert$ that contains no bad embedding and, therefore, $\Psi_{B,n}^2$ satisfies~\ref{it:disem4}.

Since $Z\in\cG_{B,F,n,p}(D,\zeta,\delta)$  there are at most $\frac{Dpn^2}{n^\delta}$ pairs of distinct edges $e_1,e_2\in E(Z)$ such that $|\cP(Z,e_1,e_2)|> \frac{D}{pn^\delta}$.
For those pairs of edges~$e_1,e_2$ we delete all embeddings~$h\in\Psi_{B,n}^2$ with~$e_1\sim_h e_2$.
Since $|\cF_{-}(Z,e)|\leq \frac{D}{p}$ for all $e\in E(Z)$ for $Z\in\cG_{B,F,n,p}(D,\zeta,\delta)$ we delete at most

\begin{align*}
\frac{Dpn^2}{n^\delta}\cdot \frac{D}{p}v(F)^2 n^{v(B)-2}&=\frac{D^2 v(F)^2n^{v(B)}}{n^\delta}=o(\vert \Psi_{B,n}\vert)
\end{align*}
embeddings from~$\Psi_{B,n}^2$. So we get for sufficiently large~$n$ a family~$\Psi_{B,n}^3\subseteq\Psi_{B,n}^2$ of size at least $\frac{1}{3}\vert \Psi_{B,n}\vert$ such that for all distinct $e_1$, $e_2\in E(Z)$
we have
\begin{enumerate}[label=\Flabel]
\item \label{it:connection1}  if $e_1\sim_h e_2$ for some $h\in\Psi_{B,n}^3$, 
	then $|\cP(Z,e_1,e_2)|\leq \frac{D}{pn^\delta}$.
\end{enumerate}
Next we will select a subset~$\Psi_{B,n}^4\subseteq\Psi_{B,n}^3$, which allows us to bound $c_{\Psi_{B,n}^4}(e_1,e_2)$ for every
pair of edges of $Z$.
For this purpose for 
\[
\eps=2\talpha=\frac{2}{13v(B)^4v(B)!}
\]
we select with repetition $\eps n^2$ times an element of~$\Psi_{B,n}^3$, where we assume for simplicity that~$\eps n^2$ is an integer.
For every selection~$S$ we define a family of embeddings~\mbox{$\Psi_S\subseteq\Psi_{B,n}^3$} by taking all embeddings that were chosen at least once in~$S$.
We will show that the random selection~$S$ a.a.s.\ satisfies that $c_{\Psi_S}(e_1,e_2)\leq \frac{1}{pn^{\delta/2}}$ for all $e_1,e_2\in E(Z)$ and that with probability less than~$\frac{1}{2}$ there are more than $\frac{\eps}{2}n^2$ embeddings that share at least two vertices with some other embedding in the selection.

First we show that a.a.s.\ $c_{\Psi_S}(e_1,e_2)\leq \frac{1}{pn^{\delta/2}}$ for all $e_1,e_2\in E(Z)$.
Since there are no bad embeddings w.r.t.\ $F$ and~$Z$ in~$\Psi_{B,n}^3$ we know that if~$e$ focuses on~$h(B)$ then~$e$ focuses on exactly one edge in~$E(h(B))$ (see property~\ref{it:bad1} in Definition~\ref{def:bad}).
Hence, for $e_1\approx_h e_2$ we may consider the following two cases.
Either~$e_1\sim_h e_2$ or $e_1$ and $e_2$ focus on two different edges in $h(B)$.

For the first case we shall use~\ref{it:connection1} and $\vert \Psi_{B,n}^3\vert\geq \frac{1}{3}\binom{n}{v(B)}$ to bound the probability that $e_1\sim_{h_i}e_2$. In fact, 
\begin{align*}
\P(e_1\sim_{h_i} e_2)&\leq \frac{D}{pn^{\delta}}\cdot v(F)^2\cdot \frac{v(B)^2\cdot (n-2) \cdots (n-v(B)+1)}{|\Psi^3_{B,n}|}\\
&\leq \frac{3D v(F)^2v(B)^2v(B)!}{ pn^{2+\delta}}\,.
\end{align*}

In the second case we shall use~\ref{it:F--3} of Definition~\ref{def:Z} for the upper bound on $|\cF_{-}(Z,e)|$.
This and the fact that two edges fix at least three vertices yield
\begin{align*}
\P(e_1\approx_{h_i} e_2 \,\text{    and not }\,e_1\sim_{h_i} e_2)
&\leq \frac{D^2}{p^2} \cdot v(F)^4\cdot \frac{v(B)^3\cdot (n-3)\cdots (n-v(B)+1)}{|\Psi^3_{B,n}|}\\
&\leq \frac{3D^2 v(F)^4v(B)^3v(B)!}{ p^{2}n^{3}}\,.
\end{align*}
Consequently
\begin{align}\label{al P e1e2hi}
\P(e_1\approx_{h_i} e_2)\leq 3D v(F)^2v(B)^2v(B)!\left(\frac{1}{pn^{2+\delta}}+\frac{Dv(F)^2v(B)}{ p^{2}n^{3}}\right)\,.
\end{align}
Since $\delta <1-\frac{1}{m_2(F)}$ we infer $n^{\delta}<C_0 n^{1-\frac{1}{m_2(F)}} < pn$ for sufficiently large~$n$.
Therefore the right hand side of~\eqref{al P e1e2hi} is of order~$\Theta(\frac{1}{pn^{2+\delta}}) $ and we can bound
\[
	\P(e_1\approx_{h_i} e_2)\leq \frac{D_0 }{pn^{2+\delta}}\,.
\]
where $D_0=4D v(F)^2v(B)^2v(B)!$.
For the expected number of connections we get
\[
	\E[c_{\Psi_S}(e_1,e_2)]\leq \sum_{i=1}^{\eps n^2}\P(e_1\approx_{h_i} e_2)\leq \frac{\eps D_0}{ pn^{\delta}}\,.
\]
Consequently,  Chernoff's Inequality yields
\[
	\P\left(c_{\Psi_S}(e_1,e_2)\geq\frac{3}{2}\cdot\frac{\eps D_0}{pn^{\delta}}\right)\leq \exp\left(-\frac{1}{12}\cdot	\frac{\eps D_0}{pn^{\delta}}\right)\,.
\]
Note that $\frac{1}{pn^{\delta}}> n^\beta$ for some $\beta>0$ since $\delta<\frac{1}{m_2(F)}$, hence, we can apply the union bound for all pairs of edges $e_1,e_2\in E(Z)$ and get that a.a.s.\ 
\[
	c_{\Psi_S}(e_1,e_2)\leq \frac{3\eps D_0}{2pn^{\delta}}\leq \frac{1}{pn^{\delta/2}}\,.
\]

Finally we verify that most pairs of selected embeddings intersect in at most one vertex.
In fact, for $i=1,\dots,\eps n^2$ let $1_{h_i}$ be the indicator random variable for the event ``there is some $j\in[\eps n^2]\setminus \{i\}$ such that $v(h_i(B)\cap h_j(B))\geq 2$'' and set $Y=\sum_{i=1}^{\eps n^2}1_{h_i}$.
Then
\[
	\E[1_{h_1}]\leq (\eps n^2-1)\frac{\binom{v(B)}{2}\cdot v(B)(v(B)-1)\cdot (n-2)\cdots (n-v(B)+1)}{\vert \Psi_{B,n}^3\vert}\leq D_1 \eps
\]
for some constant $D_1=D_1(B)$ with $0<D_1<\frac{3}{2} v(B)^4v(B)!$ independent of~$\eps$. Hence,
\[
	\E[Y]\leq \eps n^2D_1\eps=D_1\eps^2 n^2
\]
and by  Markov's Inequality we get 
\[
	\P(Y>2\E[Y])\leq \frac{1}{2}\,,
\]
so there is a selection~$S$ of~$\eps n^2$ embeddings such that $Y\leq 2D_1 \eps^2 n^2$ and $c_{\Psi_S}(e_1,e_2)\leq \frac{1}{pn^{\delta/2}}$ for all pairs of edges.
For this choice of $S$ we can simply delete all those embeddings $h_i$ that intersect with some other embedding $h_j$ in at least two vertices.
We call the remaining family~$\Psi_{B,n}^4$.
Using $D_1\leq 3v(B)^4v(B)!/2$ and the definition $\eps=2\talpha=\frac{2}{13v(B)^4v(B)!}$ yields 
\[
	|\Psi_{B,n}^4|\geq \eps n^2-2D_1\eps^2n^2\geq C\talpha n^2
\]
for some $C>1$ and, hence, $\Psi_{B,n}^4$ satisfies~\ref{it:disem1}--\ref{it:disem4}.

To achieve~\ref{it:disem5} we make use of $e(Z)\leq pn^2$ (see~\ref{it:F--1} of Definition~\ref{def:Z}).
Since no two embeddings from $\Psi_{B,n}^4$ share an edge, we may remove all embeddings from $\Psi_{B,n}^4$ which share at least one edge with $Z$ and this results in the desired family $\Xi_{B,n}^0\subseteq \Psi_{B,n}^4$ of size at least $\talpha n^2$,
which finishes the proof.
\end{proof}

For Lemma~\ref{lem:hauptlemma1} we have to show that there is a family of embeddings~$\Xi_{B,n}$ such that the hypergraph $\cH(Z,\Xi_{B,n})$ is index consistent with a profile~$\pi$. Lemma~\ref{lem:index} will ensure this.

\begin{lemma}\label{lem:index}
For all constants $1>\talpha>0$ and $D>0$, for all graphs~$F$ and~$B$ with~$F$ being strictly balanced and with $E(B)=\{e_1,\dots,e_K\}$, there exist $\alpha^\prime>0$ and $L\in\NN$ such that every graph~$Z$ on~$n$ vertices with a fixed ordering of its edge set and the property
\begin{enumerate}[label=\Zlabelb]
\item\label{prop:Z} $|\cF_{-}(Z)|\leq D n^2$
\end{enumerate}
satisfies the following.

For every $(\talpha,Z)$-normal family $\Xi_{B,n}^0$ there is an $(\alpha',Z)$-normal family
$\Xi_{B,n}\subset \Xi_{B,n}^0$ and there is a profile~$\pi$ of length at most~$L$ such that $(Z,\Xi_{B,n})$ is index consistent with profile~$\pi$.
\end{lemma}

Below we consider $Z$ and $B$ to be fixed graphs and for a simpler notation we set 
\[
	M_h=M(Z,h(B))
\]
for $h\in \Psi_{B,n}$ (see~\eqref{def:MZB} for the definition of $M(Z,h(B))$).
Note that it is rather unlikely that $M_h$ and $M_{h^\prime}$ of $\cH$ are equal for distinct $h,h^\prime\in \Xi^0_{B,n}$
and, hence, Lemma~\ref{lem:index} follows by a simple averaging argument. 
We will use Lemma~\ref{lem:index} for $Z\in \cG_{B,F,n,p}(D,\zeta,\delta)$ which satisfies~\ref{prop:Z} by~\ref{it:F--2} from Definition~\ref{def:Z}.

\begin{proof}[Proof of Lemma~\ref{lem:index}]
Let $1>\talpha>0$, $D>0$,~$F$ and~$B$ be given.
We define
\[L=(e(F)-1)\frac{2}{\talpha}v(F)^2D \qquad \text{and} \qquad \alpha^\prime=\frac{\talpha}{2L(KL)^L}\,.\]
Given some $Z$ satisfying~\ref{prop:Z} and an $(\talpha,Z)$-normal family $\Xi_{B,n}^0\subseteq\Psi_{B,n}$ we will restrict $\Xi_{B,n}^0$ to the promised set~$\Xi_{B,n}$ with the desired properties.

Note that the family  $\Xi_{B,n}\subseteq\Xi_{B,n}^0$ inherits the properties~\ref{it:disem2}--\ref{it:disem5} 
from the $(\talpha,Z)$-normality of $\Xi_{B,n}^0$ since they are independent of $\talpha$. 
Consequently, to establish that $\Xi_{B,n}$ is indeed $(\alpha^\prime,Z)$-normal,
we only have to focus on~\ref{it:disem1}. Since again property~\ref{it:normal2} of Definition~\ref{def:normal}
is inherited from the normality of $\Xi_{B,n}^0$,
it suffices to show that $\vert \Xi_{B,n}\vert \geq \alpha^\prime n^2$.

Because of~\ref{prop:Z} we know that $Z$ contains at most~$Dn^2$ copies of some~$F^\prime\subseteq F$ with $e(F^\prime)=e(F)-1$.
Also due to $\Xi_{B,n}^0$ being $(\talpha,Z)$-normal (see~\ref{it:disem4}) there are no bad embeddings w.r.t.\ $F$ and~$Z$ in~$\Xi_{B,n}^0$ and thus by Fact~\ref{fact:bad imp reg} the pair~$(Z,\Xi_{B,n}^0)$ is regular.
In particular, for every $h\in \Xi_{B,n}^0$ we have that
 every edge~$e\in M_h$ focuses on exactly one $b\in E(h(B))$.
Furthermore, since every $h\in \Xi_{B,n}^0$ also does not satisfy~\ref{it:bad3} of Definition~\ref{def:bad}, each $e\in M_h$ focuses on one $b\in E(h(B))$ in only one way, i.e.\ there is only one copy of~$F$ in $Z\cup h(B)$ containing~$b$ and~$e$.
Therefore,~$\l_h=|M_h|$ is a multiple of $e(F)-1$ and each~$M_h$ gives rise to $\l_h/(e(F)-1)$ copies of graphs~$F^\prime$ in~$Z$, where each such $F'$ is
obtained from~$F$ by removing some edge.
Clearly, each such $(e(F)-1)$-element subset of~$M_h$ might be completed to a copy of~$F$ in at most $\binom{v(F)}{2}-e(F)+1<v(F)^2$ ways.

 Applying the upper bound on the number of copies of~$F$ with one edge removed from~\ref{prop:Z} yields
 \[
	\sum_{h\in \Xi^0_{B,n}}\frac{\l_h}{e(F)-1}\leq v(F)^2 \cdot Dn^2\,.
\]
So there are at most $\talpha n^2/2$ embeddings $h\in \Xi_{B,n}^0$ with $\l_h> L$, and, consequently, at least $\talpha n^2/2$ embeddings~$h\in\Xi_{B,n}^0$ with~$\l_h\leq L$.
Since there are at most~$K^\l$ different profiles of length~$\l$,
there must be a profile~$\pi$ of length $\l\leq L$ and a subset $\Xi^\prime_{B,n}\subseteq \Xi_{B,n}^0$ with
\[
	\vert \Xi^\prime_{B,n}\vert\geq \frac{1}{LK^L}\cdot \frac{\talpha}{2}n^2
\]
such that $(Z,\Xi^\prime_{B,n})$ has profile~$\pi$.

Next we apply another averaging argument to achieve index consistency.
We consider some partition $Z_1\dcup\dots\dcup Z_\l$ of~$Z$ into~$\l$ classes chosen uniformly at random.
Recall that we ordered the edges of~$Z$.
For $h\in \Xi^\prime_{B,n}$ consider $M_h=(z_1,\dots,z_\l)$ with the inherited ordering of~$Z$.
We include~$h$ in~$\Xi_{B,n}$ if $z_i\in Z_i$ for all $i=1,\dots,\l$.
Clearly $\P(h\in \Xi_{B,n})=\frac{1}{\l^\l}$ and $\E[\vert \Xi_{B,n}\vert]=\frac{\vert \Xi^\prime_{B,n}\vert}{\l^\l}$, which means there is an $\Xi_{B,n}\subseteq \Xi^\prime_{B,n}$ with 
\[
\vert \Xi_{B,n}\vert \geq \vert \Xi^\prime_{B,n} \vert/\l^\l\geq \frac{1}{L^L} \frac{\talpha n^2}{2L K^L}=\alpha^\prime n^2\,.
\]
Now let $h,h^\prime\in \Xi_{B,n}$ and let $z\in M_h\cap M_{h^\prime}$.
Since $z\in Z_j$ for some partition class $Z_j$ we know that~$z$ has index~$j$ in both $M_h$ and $M_{h^\prime}$.
Therefore $(Z,\Xi_{B,n})$ is index consistent which finishes the proof.
\end{proof}

\subsection{Proof of Lemma~\ref{lem:hauptlemma1}}\label{subsec:Proofhauptlemma}

Finally we prove Lemma~\ref{lem:hauptlemma1}.
The previous lemmas will be utilised to show that the hypergraph $\cH(Z,\Xi)$ satisfies the conditions of
Theorem~\ref{thm:ST} of Saxton and Thomason about independent sets in hypergraphs.

\begin{proof}[Proof of Lemma~\ref{lem:hauptlemma1}]
Let constants $C_1>C_0>0$, $\frac{1}{3}>\alpha>0$ and graphs~$F$ and~$B$ with~$F$ being strictly balanced be given.

First we fix all constants used in the proof. For the given graphs $F$ and~$B$ and the given constants $C_1$ and~$C_0$ 
Lemma~\ref{lem:G(BFnzg)} yields constants~$D>0$, $\zeta>0$, and~$\delta$ 
with $0<\delta<\min\big\lbrace\tfrac{1}{m_2(F)}, 1-\tfrac{1}{m_2(F)}\big\rbrace$.
Similarly Lemma~\ref{lem:index} applied to~$F$,~$B$, $D$ and
\[\talpha=\frac{1}{13v(B)^4v(B)!}\]
yields $\alpha^\prime$ and~$L$.
Fixing an auxiliary constant
\[
	k=\binom{L}{e(F)-1}\binom{v(F)}{2}
\]
allows us to set
\begin{align}\label{eq:gamma beta}
\beta=\frac{\alpha^\prime}{Dkv(F)^2}\qand\gamma=\frac{\delta}{10L}\,.
\end{align}
We shall show that $\alpha'$, $\beta$, $\gamma$, and $L$ defined this way have the desired property. For that 
let $p=p(n)=c(n)n^{-1/m_2(F)}$ for some $c(n)$ satisfying $C_0\leq c(n)\leq C_1$. We shall show that 
$G(n,p)$ a.a.s.\ satisfies the property of Lemma~\ref{lem:hauptlemma1}. Hence, in view of Lemma~\ref{lem:G(BFnzg)} 
we may assume that the graphs $Z$ considered in Lemma~\ref{lem:hauptlemma1} are 
from the set $\cG_{B,F,n,p}(D,\zeta,\delta)$. Moreover, let~$n$
be sufficiently large, so that Lemma~\ref{lem:disEmb} applied 
with~$F$,~$B$,~$\alpha$,~$D$,~$\zeta$,~$\delta$,~$C_1$ and~$C_0$ holds for $n$.

Now let $Z\in \cG_{B,F,n,p}(D,\zeta,\delta)$ such that for~$h\in\Psi_{B,n}$ chosen uniformly at random 
we have \[\P(Z\cup h(B) \rightarrow (F)_2^e)>1-\alpha.\]
Then Lemma~\ref{lem:disEmb} yields an $(\talpha,Z)$-normal family 
of embeddings~$\Xi_{B,n}^0\subseteq \Psi_{B,n}$, i.e.,  the family~$\Xi_{B,n}^0$ 
satisfies properties \ref{it:disem1}--\ref{it:disem5} of Lemma~\ref{lem:disEmb} 
for the parameters chosen above.

Since $Z\in \cG_{B,F,n,p}(D,\zeta,\delta)$ it satisfies property~\ref{it:F--2} of Definition~\ref{def:Z}
and, hence,~$Z$ satisfies in particular
assumption~\ref{prop:Z} of Lemma~\ref{lem:index}. Consequently, 
Lemma~\ref{lem:index} yields an $(\alpha',Z)$-normal family $\Xi_{B,n}\subseteq \Xi_{B,n}^0$
and a profile~$\pi$ of length~$\l\leq L$ such that the pair $(Z,\Xi_{B,n})$ is index consistent for~$\pi$.

Next we consider the hypergraph $\cH=\cH(Z,\Xi_{B,n})$ defined by
\[	
	V(\cH)=E(Z)
\qand 
	E(\cH)=\lbrace M(Z,h(B))\colon h\in \Xi_{B,n}\rbrace\,,
\]
where 
\[
	M(Z,h(B))=\lbrace z\in E(Z)\colon \text{there is }b\in E(h(B))\text{ such that }z\text{ focuses on }b \rbrace\,.
\]
Clearly, $\cH$ is an $\l$-uniform hypergraph on~$m=e(Z)$ vertices.
Below we show that~$\cH$ satisfies the assumptions of Theorem~\ref{thm:ST} for
\[
	\eps=\tfrac{1}{4} \qqand \tau=n^{-\frac{\delta}{4(\l-1)}}\,.
\]
Since $Z\in\cG_{B,F,n,p}(D,\zeta,\delta)$ it displays properties~\ref{it:F--1}--\ref{it:F--B} of 
Definition~\ref{def:Z}. In particular, the property~\ref{it:F--1} guarantees
\begin{align}\label{eq:m}
\frac{1}{4}pn^2\leq e(Z)= m\leq pn^2<n^2\,.
\end{align}

Now we bound~$e(\cH)$.
Since $\Xi_{B,n}$ is $\alpha'$-normal, it follows from~\ref{it:normal1} and~\ref{it:normal2} of Definition~\ref{def:normal}  that $\alpha^\prime n^2\leq |\Xi_{B,n}|\leq n^2$ and, consequently, we have
 $e(\cH)\leq n^2$.
On the other hand, for any hyperedge~$M_h$ of size $\l$ there are at most $\binom{\l}{e(F)-1}$ different copies of some~$F^\prime\subseteq F$ with $e(F^\prime)=e(F)-1$ in~$M_h$ and each such copy can be extended to~$F$ by at most $\binom{v(F)}{2}$ different boosters since all boosters are edge disjoint.
Consequently,~$M_h$ could be the hyperedge for at most $\binom{\l}{e(F)-1}\binom{v(F)}{2}\leq k$ different embeddings~$h\in \Xi_{B,n}$
and, therefore, we have
\begin{equation}\label{eq:eH}
	\frac{\alpha^\prime n^2}{k}\leq e(\cH) \leq n^2\,.
\end{equation}
Hence, for the average degree of~$\cH$ we obtain
\[
	d(\cH)=\l\cdot\frac{ e(\cH)}{v(\cH)}\geq\l\cdot\frac{ \alpha^\prime n^2}{k}\cdot\frac{1}{pn^2}=\frac{\l\alpha^\prime}{kp}\,.
\]

We denote by $\Delta_1(\cH)=\max_{v\in V(\cH)}\vert\{e\in E(\cH):e \text{ contains }v\}\vert$ the maximum vertex degree and by
$\Delta_2(\cH)=\max_{(v,v^\prime)\in \binom{V(\cH)}{2}}\vert\{e\in E(\cH):e \text{ contains }v\text{ and }v^\prime \}\vert$ the maximum codegree of~$\cH$ and below 
we will bound~$\Delta_1(\cH)$ and~$\Delta_2(\cH)$.

We start with $\Delta_1(\cH)$. Suppose $e\in M(Z,h(B))$ for some 
$h\in\Xi_{B,n}$. Since $\Xi_{B,n}$ contains no bad embeddings w.r.t.\ $F$ and~$Z$ and $E(h(B))\cap E(Z)=\emptyset$
there exists a unique copy $F_-\in\cF_-(Z,e)$ with $e\in E(F_-)$ and $f\in h(B)$ such that $F_-+f$ forms a copy of $F$.
Moreover, since every two distinct embeddings $h$, $h'\in  \Xi_{B,n}$ intersect in at most one vertex the degree of
$e$ in $\cH$ is bounded by $|\cF_-(Z,e)|\cdot\binom{v(F)}{2}$.
Consequently, it follows from property~\ref{it:F--3} given by $Z\in\cG_{B,F,n,p}(D,\zeta,\delta)$  that
\[
	\Delta_1(\cH)\leq \frac{D}{p}\cdot\binom{v(F)}{2}\,.
\]

For~$\Delta_2(\cH)$ we have to look at pairs of edges of~$Z$.
Two edges $e_1,e_2\in E(Z)$ are both contained in $M(Z,h(B))$ if and only if~$e_1\approx_h e_2$.
By~\ref{it:disem3} we know $c_{\Xi_{B,n}}(e_1,e_2)\leq \frac{1}{pn^{\delta/2}}$, so 
\[
	\Delta_2(\cH)\leq \frac{1}{pn^{\frac{\delta}{2}}}\,.
\]
Note that~$pn^{\delta/2}\rightarrow 0$ for $n\rightarrow \infty$ since $\delta\leq\frac{1}{m_2(F)}$.

In order to verify the assumptions of Theorem~\ref{thm:ST} we estimate~$\delta(\cH,\tau)$ for $\eps$ and $\tau$
defined above. Indeed we have
\begin{align*}
\delta(\cH,\tau)
&= 2^{\binom{\l}{2}-1}\sum_{j=2}^\l 2^{-\binom{j-1}{2}}\frac{1}{\tau^{j-1}md(\cH)}\sum_{v\in V(\cH)}d^{(j)}(v)\\
&\leq 2^{\binom{\l}{2}-1}\sum_{j=2}^\l 2^{-\binom{j-1}{2}}\frac{1}{\tau^{j-1}md(\cH)}\cdot m\cdot \Delta_2(\cH)\\
&\leq 2^{\binom{\l}{2}-1}\sum_{j=2}^\l \frac{1}{\tau^{\l-1}d(\cH)}\cdot \Delta_2(\cH)\\
&\leq 2^{\binom{\l}{2}-1}\cdot\l\cdot n^{\frac{\delta}{4}}\cdot \frac{kp}{\l\alpha^\prime}\cdot \frac{1}{pn^{\frac{\delta}{2}}}\\
&= 2^{\binom{\l}{2}-1}\cdot \frac{k}{\alpha^\prime}\cdot \frac{1}{n^{\frac{\delta}{4}}}\\
&\leq \frac{\eps}{12\l!}\,,
\end{align*}
where the last inequality holds for sufficiently large~$n$.

By Theorem~\ref{thm:ST} there exist some constant $c=c(\l)$ and a family $\cJ\subset \powerset(V(\cH))$ satisfying~\ref{it:ST1},~\ref{it:ST2} and~\ref{it:ST3} from Theorem~\ref{thm:ST}. We
define 
\[\cC=\{C\subset V(\cH)\colon C=V(\cH)\setminus J \text{ for one }J\in \cJ\}\,.\]
Below we show that~$\cC$ has the desired properties~\eqref{it:lem:core1},~\eqref{it:lem:core2} and~\eqref{it:lem:core3} of Lemma~\ref{lem:hauptlemma1}.

\eqref{it:lem:core1} follows from~\ref{it:ST3} since $\vert \cC\vert=\vert \cJ\vert$ and
\[
	\log \vert \cJ\vert \leq c\tau \log(1/\tau)\log(1/\eps)m\leq m \cdot n^{-\frac{\delta}{4(\l-1)}} c\log(1/\tau)\log(1/\eps)\leq m^{1-\gamma}\,,
\]
where the last inequality follows for sufficiently large~$n$ from 
\[
	m^\gamma \overset{\eqref{eq:m}}{<}n^{2\gamma}\overset{\eqref{eq:gamma beta}}{\leq} n^{\frac{\delta}{5\l}}\,,
\]
since $c=c(\l)$ and $\log(1/\eps)$ are constants independent of $n$
and $\log(1/\tau)<\log n$.

\eqref{it:lem:core2} follows from~\ref{it:ST2}.
Assume for a contradiction that there is $C\in\cC$ with $\vert C\vert <\beta m$ and let~$J=V\setminus C\in \cJ$.
Then we count the number of hyperedges of~$\cH$.

\begin{align*}
e(\cH)&\overset{\phantom{\eqref{eq:gamma beta}}}{\leq} e(\cH[V\setminus C])+\vert C \vert \cdot \Delta_1(\cH)\\
&\overset{\phantom{\eqref{eq:gamma beta}}}{<}e(\cH[J])+\beta m \cdot \frac{D}{p}\binom{v(F)}{2}\\
&\overset{\eqref{eq:m}}{\leq} \eps e(\cH)+\beta D \binom{v(F)}{2}n^2\\
&\overset{\eqref{eq:eH}}{\leq} \eps e(\cH)+\frac{\beta D k}{\alpha^\prime}\binom{v(F)}{2}e(\cH)\\
&\overset{\phantom{\eqref{eq:gamma beta}}}{=}\left(\eps+\frac{\beta D k}{\alpha^\prime}\binom{v(F)}{2}\right)e(\cH)\\
&\overset{\eqref{eq:gamma beta}}{<} e(\cH)
\end{align*}
with a contradiction, so $\vert C\vert \geq\beta m $ for all $C\in\cC$.

\eqref{it:lem:core3} For a hitting set~$A$ of~$\cH$ consider the independent set $I=V\setminus A$.
Hence by~\ref{it:ST1} of Theorem~\ref{thm:ST} there exists $J\in\cJ$ such that $I\subseteq J$ and, therefore, $A\supseteq V\setminus J=C$ which is an element of~$\cC$.
\end{proof}

\section{Proof of Lemma~\ref{lem:base}}\label{sec:proof Lemma 2}
The proof of Lemma~\ref{lem:base} follows the proof in~\cite{FRRT06}*{Lemma~2.3} and is based on an application of the regularity method for subgraphs of sparse random graphs which we introduce first.

Let $\eps>0$, $p\in (0,1]$ and $H=(V,E)$ be a graph.
For $X$, $Y\subset V$ non-empty and disjoint let
\[d_{H,p}(X,Y)=\frac{e(X,Y)}{p\vert X\vert \vert Y\vert}\]
and we say $(X,Y)$ is $(\eps,p)$-regular if
\[\vert d_{H,p}(X,Y)-d_{H,p}(X^\prime,Y^\prime)\vert <\eps\]
for all subsets $X^\prime\subseteq X$ and $Y^\prime\subseteq Y$ with $\vert X^\prime\vert \geq \eps \vert X\vert$ and $\vert Y^\prime\vert \geq \eps \vert Y\vert$.
We will use the sparse regularity lemma in the following form (see, e.g.,~\cite{Ko97}).

\begin{lemma}\label{lem:SpRL}
For all $\eps>0$ and $t_0$ there exists an integer~$T_0$ such that for every function $p=p(n)\gg 1/n$ a.a.s.\ $G\in G(n,p)$ has the following property.
Every subgraph $H=(V,E)$ of~$G$ with $\vert V\vert =n$ vertices admits a partition $V=V_1\dcup\dots\dcup V_t$ satisfying
\begin{enumerate}[label=\rmlabel]
\item $t_0\leq t\leq T_0$,
\item $\vert V_1\vert \leq \dots \leq \vert V_t\vert \leq \vert V_1\vert +1$  and
\item all but at most $\eps t^2$ pairs $(V_i,V_j)$ with $i\neq j$ are $(\eps,p)$-regular.\qed
\end{enumerate}
\end{lemma}

For a partition~$\cP$ as in the last lemma we call the graph~$R=R(\cP,d,\eps)$ with vertex set $V(R)=\{V_1,\dots ,V_t\}$ and edges
\[
	\{V_i,V_j\}\in E(R)\ \Longleftrightarrow\ (V_i,V_j) \text{ is }(\eps,p)\text{-regular with}\ d_{H,p}(V_i,V_j)\geq d
\]
the reduced graph w.r.t.\ $\cP$,~$d$, and $\eps$.

The next lemma is a counting lemma for subgraphs of random graphs 
from~\cites{BMS12,CGSS14,ST12}. For the proof of Lemma~\ref{lem:base}
we only need this (and the following lemma) for fixed bipartite graphs. 
However, we state those auxiliary lemmas in its general form.
\begin{lemma}\label{lem:SpCL}
For every graph~$F$ with vertex set $V(F)=[\l]$ and $d>0$ there exist $\eps>0$ and $\xi>0$ such that for every $\eta >0$ there exists $C>0$ such that for $p>Cn^{-1/m_2(F)}$ a.a.s.\ $G\in G(n,p)$ satisfies the following.

Let $H=(V_1\dcup\dots\dcup V_\l, E_H)$ be an $\l$-partite (not necessarily induced) subgraph of~$G$ with vertex classes
of size at least $\eta n$ and with the property that for every edge $\lbrace i,j\rbrace\in E(F)$ the pair $(V_i,V_j)$ in~$H$ is $(\eps,p)$-regular with density $d_{H,p}(V_i,V_j)\geq d$.
Then the number of partite copies of~$F$ in~$H$ is at least
\[
	\xi p^{e(F)}\prod_{i=1}^\l \vert V_i\vert\,,
\]
where a partite copy is a graph homomorphism~$\vphi\colon F\rightarrow H$ with $\vphi(i)\in V_i$.\qed
\end{lemma}

The next lemma bounds the number of edges between large sets of vertices of $G(n,p)$ 
as well as the number of copies of some bipartite graphs~$F^\star$ with two vertices from a prescribed set~$W$. 

\begin{lemma}\label{lem:EigGnp}
Let~$F^\star$ be a graph with two vertices $a_1,a_2\in V(F^\star)$ with $a_1a_2\not\in E(F^\star)$.
For all $(\log n)/n \leq p=p(n)<1$
the random graph $G\in G(n,p)$ satisfies a.a.s.\ the following properties.

\begin{enumerate}[label=\Alabel]
\item\label{it:A2} For all disjoint subsets $U$, $W\subseteq V(G)$ 
with $\vert U \vert, |W| \geq n/\log \log n$ we have
\[
	p|U|^2/3< e_G(U)<p\vert U\vert^2
	\qand
	p|U||W|/2<e_G(U,W)<2p|U||W|\,.	
\]
\item\label{it:A3} For all subsets $W\subset V(G)$
there exists a set of edges $E_0\subseteq E(G)$ with $\vert E_0\vert =n\log n$
such that there are at most $2p^{e(F^\star)} n^{v(F^\star)-2}|W|^{2}$ many copies~$\vphi(F^\star)$ of~$F^\star$ in the graph $(V(G),E(G)\setminus E_0)$
with $V(\vphi(F^\star))\cap W=\{\vphi(a_1),\vphi(a_2)\}$.
\end{enumerate}
\end{lemma}
The proof of~\ref{it:A2} follows directly from  Chernoff's inequality and 
the proof of~\ref{it:A3} is based on the so-called \emph{deletion method} in form of the following lemma.
\begin{lemma}\label{lem:JLRRandomgraphsL2.51}~\cite{JLR00}*{Lemma~2.51}
Let~$\Gamma$ be a set, $S\subseteq [\Gamma]^s$ and $0<p<1$.
Then for every~$k>0$ with probability at least $1-\exp(-\frac{k}{2s})$
there exists a set $E_0\subset \Gamma_p$ of size~$k$ such that $\Gamma_p\setminus E_0$ contains at most $2\mu$ sets from~$S$ where
$\mu$ is the expected number of sets from~$S$ contained in~$\Gamma_p$. \qed
\end{lemma}

\begin{proof}[Proof of Lemma~\ref{lem:EigGnp}]
Since part~\ref{it:A2} follows from Chernoff's inequality, we will only focus on
property~\ref{it:A3}, which is a direct consequence of Lemma~\ref{lem:JLRRandomgraphsL2.51}.

In fact, let $V$ be a set of $n$ vertices,~$W\subset V$ and a graph~$F^\star$ with
two fixed vertices $a_1,a_2\in V(F^\star)$ not forming an edge in $F^\star$.
We use Lemma~\ref{lem:JLRRandomgraphsL2.51} with $\Gamma=\binom{V}{2}$, $s=e(F^\star)$,
\[
	S=\big\{\text{copies } \vphi(F^\star)\text{ of }F^\star\text{ in }(V,\Gamma)\text{ with }V(\vphi(F^\star))\cap W=\{\vphi(a_1),\vphi(a_2)\}\big\}\,,
\]
$p$, and $k=n\log n$. In particular, $\Gamma_p=G(n,p)$ in our setup here. 
With probability at least $1-\exp\big(-\frac{n\log n}{2e(F^\star)}\big)$ there exists a set $E_0\subseteq E(G(n,p))$ of size at most $n\log n$ such that there are at most
\[2\mu\leq 2p^{e(F^\star)} n^{v(F^\star)-2}|W|^2\]
many copies~$\vphi(F^\star)$ with $V(\vphi(F^\star))\cap W=\{\vphi(a_1),\vphi(a_2)\}$ in $(V,E(G(n,p))\setminus E_0)$.
The lemma then follows from the union bound applied for all $2^n$ possible choices $W\subset V$.
\end{proof}

Finally, we can prove Lemma~\ref{lem:base}.
Let $F$ be a strictly balanced and nearly bipartite graph. Let~$G$ be a typical graph  (with respect to the properties of Lemmas~\ref{lem:SpRL}--\ref{lem:EigGnp}) in~$G(n,p)$ and let~$H$ be a subgraph of~$G$ with $\vert E(H)\vert\geq \lambda \vert E(G)\vert$.
First we apply the sparse regularity lemma (Lemma~\ref{lem:SpRL}) to~$H$.
Since~$H$ is relatively dense in $G(n,p)$ we infer that the corresponding reduced graph~$R$ (for suitable chosen
parameters) has many, i.e.\ $\Omega(|V(R)|^2)$ edges.
So we can find many large complete bipartite graphs in~$R$.
We conclude that there is some partition class $V_i\in V(R)$ contained in many complete bipartite graphs.

We analyse the graph~$G_0=\Base_H(F)[V_i]$ on the vertex set~$V_i$ with edges being those pairs in $\binom{V_i}{2}$ that 
complete a copy of the bipartite graph $F'\subseteq F'+e=F$ in $H$ to a copy of $F$.
We say that~$G_0$ is $(\rho,d)$-dense if for all $W\subseteq V(G_0)$ with $\vert W\vert\geq \rho |V_i|$ we have 
$e_{G_0}(W)\geq d\binom{\vert W\vert}{2}$.
It is well known that sufficiently large $(\rho,d)$-dense graphs contain any fixed subgraph (see e.g.\ \cite{RR95}). 
\begin{lemma}\label{lem:rddense}
For all $d>0$ and~$F$ there exist~$\rho$, $c_0>0$ and $n_0\in\NN$ such that every $(\rho,d)$-dense graph~$G_0$ with $v(G_0)=n\geq n_0$ contains at least $c_0 n^{v(F)}$ copies of~$F$.\qed
\end{lemma}
To show the $(\rho,d)$-denseness of~$G_0$ we consider $W\subseteq V_i$ with $\vert W\vert\geq \rho\vert V_i\vert$.
Then by Lemma~\ref{lem:SpCL} we will find many copies of~$F'$ in~$H$ where the missing edge has to be in~$\binom{W}{2}$.
Together with an upper bound for the number of graphs that are combinations of two different copies of~$F'$ (\ref{it:A3} of Lemma~\ref{lem:EigGnp}) we ensure that not too many copies of~$F'$ are completed to $F$ by the same pair in~$W$.
Thus there are many edges in $\Base_H(F)[W]$ and~$G_0$ is $(\rho,d)$-dense.

\begin{proof}[Proof of Lemma~\ref{lem:base}]
Let~$\lambda>0, C_1>C_0>0$ and let~$F$ be a strictly balanced nearly bipartite graph such that $F=F^\prime+\lbrace a_1,a_2\rbrace$, where~$F^\prime$ is bipartite with partition classes $A=\lbrace a_1,\dots,a_a\rbrace$ and $B=\lbrace b_1,\dots,b_b\rbrace$.

The Sparse Counting Lemma (Lemma~\ref{lem:SpCL}) applied with~$F^\prime$ and $d_{\textrm{CL}}=\lambda/4$ yields constants $\eps_{\textrm{CL}}>0$ and $\xi_{\textrm{CL}}>0$.
Since we don't know whether the given constant $C_0$ is at least $1$ or not, we find it convenient to
fix an auxiliary constant
\begin{equation}\label{eq:C0'}
C'_0=\min\{1,C_0^{e(F)-1}\}\,.
\end{equation}
Furthermore, we set
\begin{equation}
d=\frac{\left(\lambda/6\right)^{2(a-1)b}\cdot \xi_{\textrm{CL}}^2\cdot C_0^{2(e(F)-1)}\cdot C'_0}{64\cdot a^{2a}b^{2b}\cdot(v(F)+1)^{v(F)}\cdot C_1^{2(e(F)-1)}}\,.\label{eq:d}
\end{equation}
Next we appeal to Lemma~\ref{lem:rddense}. For $F$ and for this choice of~$d$ this lemma yields constants $\rho$, 
$c_0>0$ and $n_0\in\NN$. Furthermore, set
\begin{equation}
\eps=\min\left\lbrace\frac{\rho\eps_{\textrm{CL}}}{4},\frac{\lambda}{48}\right\rbrace \qqand t_0= \frac{48}{\lambda}ab\,. \label{eq:eps und t}
\end{equation}
Lemma~\ref{lem:SpRL} applied with~$\eps$ and~$t_0$ yields $T_0\in \NN$ and  Lemma~\ref{lem:SpCL}
applied with  $\eta_{\textrm{CL}}=\rho/(2T_0)$ yields~$C_{\textrm{CL}}$.
Finally, we fix the promised 
\[
	\eta=c_0T_0^{-v(F)}
\]
and let $C_0 n^{-1/m_2(F)}\leq p=p(n) \leq C_1n^{-1/m_2(F)}$.
For later reference we note that due to the balancedness of~$F$ we have
\begin{equation}\label{eq:expF}
	p^{e(F)}n^{v(F)}\leq C_1^{e(F)-1}pn^2
\end{equation}
and owing to the 
choice of $C_0'$ in~\eqref{eq:C0'} we have 
\begin{equation}\label{eq:minexp}
	p^{e(F_1)}n^{v(F_1)}\geq C'_0pn^2
\end{equation}
for every subgraph $F_1\subseteq F$ with $e(F_1)\geq 1$. Moreover, since we applied 
Lemma~\ref{lem:SpCL} for $F'\subsetneq F$, the strict balancedness 
of~$F$ implies $m_2(F)>m_2(F')$. Consequently,
for sufficiently large $n$ we have
\[
	C_\textrm{CL}n^{-1/m_2(F')}\leq C_0 n^{-1/m_2(F)}\leq p\,.
\]

Since we have to show that $G(n,p)$ a.a.s.\ satisfies $T(\lambda, \eta,F)$ we can assume that~$n$ is arbitrarily large.
Consider any $G\in G(n,p)$ that satisfies the properties of Lemma~\ref{lem:SpRL} and Lemma~\ref{lem:SpCL}, as well as property~\ref{it:A2} and property~\ref{it:A3} of Lemma~\ref{lem:EigGnp} for all bipartite graphs~$F^\star$ such that~$F^\star$ is the union of two different copies~$\vphi_1(F')$ and $\vphi_2(F')$ of~$F'$ with $\{\vphi_1(a_1),\vphi_1(a_2)\}=\{\vphi_2(a_1),\vphi_2(a_2)\}$. 
In other words, for the rest of the proof we consider a fixed graph $G$ to which we can apply the
Lemmas~\ref{lem:SpRL}--\ref{lem:EigGnp} and we will show that such a $G$ satisfies~$T(\lambda, \eta,F)$.
For that let $H\subseteq G$ with
\[
	e(H)\geq \lambda e(G)>\frac{1}{3}\lambda pn^2
\]
where the second inequality follows from property~\ref{it:A2}  of Lemma~\ref{lem:EigGnp}.

Lemma~\ref{lem:SpRL} applied to~$H$ yields a partition~$\cP$ of the vertices $V=V_1\dcup\dots\dcup V_t$ with at least $(1-\eps)\binom{t}{2}$ many $(\eps,p)$-regular pairs for some $t$ with $t_0\leq t\leq T_0$.
We assume w.l.o.g.\ that~$t$ divides~$n$.
We infer that there are at least $\frac{\lambda}{6}\binom{t}{2}$ regular pairs with edge density at least $\frac{\lambda}{4}p$ since otherwise we could bound the number of edges of~$H$ by
\begin{align*}
e(H)&\overset{\phantom{\eqref{eq:eps und t}}}{\leq} \frac{\lambda}{6}\binom{t}{2}\cdot 2p\left(\frac{n}{t}\right)^2 + \binom{t}{2} \cdot \frac{\lambda}{4} p \left(\frac{n}{t}\right)^2 + \eps\binom{t}{2}\cdot 2p\left(\frac{n}{t}\right)^2 + t\cdot p\left(\frac{n}{t}\right)^2\\
&\overset{\phantom{\eqref{eq:eps und t}}}{\leq}\frac{1}{2}pn^2\left( \frac{\lambda}{3}+\frac{\lambda}{4}+2\eps+\frac{2}{t}\right)\\
&\overset{\eqref{eq:eps und t}}{\leq} \frac{1}{3}\lambda pn^2\,,
\end{align*}
which would contradict the derived lower bound $e(H)> \frac{1}{3}\lambda pn^2$.

Let~$R=R(\cP,d_{\textrm{CL}},\eps)$ be the reduced graph w.r.t.\ the partition $\cP$ and relative 
density $d_{\textrm{CL}}=\frac{\lambda}{4}$. In particular~$R$ has exactly~$t\geq t_0$ vertices and 
at least $\frac{\lambda}{6}\binom{t}{2}$ edges. 
It follows from the theorem of K{\H{o}}v{\'a}ri, S{\'o}s and Tur{\'a}n~\cite{KST54} (see, e.g.,~\cite{ES84}*{Lemma~1}) that there are at least $\gamma t^{a+b-1}$ copies of the complete bipartite graph~$K_{a-1,b}$ in~$R$ where\footnote{Strictly speaking, 
in~\cite{KST54} no such lower bound on the number of copies of 
complete graphs in dense large graphs is given. However, the proof from~\cite{KST54} combined with standard convexity arguments gives the bound stated here
and such an argument can be found for example in~\cite{ES84}*{Lemma~1}.}
\begin{equation}\label{eq:gamma}
\gamma=\gamma(F,\lambda)=\frac{1}{2}\frac{1}{(a-1)^{a-1}b^b}\left(\frac{\lambda}{6}\right)^{(a-1)b}\,.
\end{equation}
Hence, there is a partition class $V_{a_0}$ of $\cP$ such that~$V_{a_0}$ is contained in at least $\gamma t^{a+b-2}$ copies of~$K_{a-1,b}$ in~$R$ where $V_{a_0}$ is always contained in partition class $A$ of $K_{a-1,b}$ for these copies. 

Our goal is to show that the graph~$G_0$ induced by $\Base_F(H)$ on~$V_{a_0}$ is $(\rho,d)$-dense, which due to our choice of $c_0$ and $\eta$ above leads to $c_0(n/t)^{v(F)}>\eta n^{v(F)}$ copies of~$F$ in~$G_0$ (see Lemma~\ref{lem:rddense}).
So let $W\subseteq V_{a_0}$ with $|W|\geq \rho |V_{a_0}|$ and fix some partition  
$W=W_1\dcup W_2$ with $|W_1|=|W_2|=|W|/2$ (for simplicity, we may assume that $|W|$ is even). 
Note that for any $j$ for which  $(V_{a_0},V_j)$ is 
$(\eps,p)$-regular we still have that $(W_1,V_j)$ and $(W_2,V_j)$ are $(2\eps/\rho,p)$-regular.

We will ensure many copies of~$F^\prime$ with $a_1\in W_1$ and $a_2\in W_2$ which force edges in $G_0=\Base_F(H)[V_{a_0}]$.
However, we have to make sure that not too many copies force the same edge in~$G_0$.
For this purpose we delete some edges by~\ref{it:A3} of Lemma~\ref{lem:EigGnp} to restrict the number of graphs $F^\star$ 
that are unions of two different copies of~$F^\prime$ that force the same edge in~$G_0$.

Let $\vphi_1(F^\prime)$ and $\vphi_2(F^\prime)$ be two copies of $F^\prime$ satisfying $\vphi_1(\{a_1,a_2\})=\vphi_2(\{a_1,a_2\})$ and let~$F^\star=\vphi_1(F^\prime)\cup \vphi_2(F^\prime)$.
We find by~\ref{it:A3} of Lemma~\ref{lem:EigGnp} at most $n\log n$ edges $E_{F^\star}$ such that there are at most
\begin{equation}
2p^{e(F^\star)}n^{v(F^\star)-2}|W|^2\label{eq:number two Fe}
\end{equation}
copies of~$F^\star$ in $(V(H),E(H)\setminus E_{F^\star})$ with $\vphi_1(a_1),\vphi_1(a_2)\in W_1\cup W_2$.
We repeat this argument for all possible graphs $F^\star$ that can be created this way and we denote by $\cF^{\star}$ the family of those graphs.
Since there are at most $2(a+1)^{a-2}(b+1)^b$ such graphs $F^{\star}$, in total we delete at most
\[
	2(a+1)^{a-2}(b+1)^bn\log n=o(pn^2)
\]
edges of~$H$, i.e., for $H'=H-\bigcup_{F^\star\in\cF^{\star}}E_{F^\star}$ we have
\[
	e(H') \geq (1-o(1))e(H).
\]
In particular,  for sufficiently large~$n$ the density and the regularity of the pairs in the partition~$\cP$ is not affected much and $(\delta,p)$-regular pairs in $H$ are still $(2\delta,p)$-regular in $H'$.

Lemma~\ref{lem:SpCL} yields many copies of~$F^\prime$ in~$H^\prime$.
In fact, since $m_2(F^\prime)< m_2(F)$ we get
\[
	p\geq C_0 n^{-\frac{1}{m_2(F)}}> C_{\textrm{CL}}n^{-\frac{1}{m_2(F^\prime)}}\,.
\]
For any copy of $K_{a-1,b}$ in the reduced graph~$R$ that contains~$V_{a_0}$ among the $a-1$ classes of the bipartition of $K_{a-1,b}$ Lemma~\ref{lem:SpCL} applied with $\eps_\textrm{CL}\geq 4\eps/\rho$ 
(see~\eqref{eq:eps und t}) yields at least
\[
	\xi_{\textrm{CL}}p^{e(F)-1} \left(\frac{n}{t}\right)^{v(F)-2}\vert W_1\vert \vert W_2\vert
	=
	\frac{1}{4}\xi_{\textrm{CL}}p^{e(F)-1}\left(\frac{n}{t}\right)^{v(F)-2}|W|^2
\]
partite copies of~$F^\prime$ in~$H^\prime$ with $a_1\in W_1$ and $a_2\in W_2$.
Repeating this for the $\gamma t^{a+b-2}$ different copies of $K_{a-1,b}$ in~$R$ that contain $V_{a_0}$ in the described way,
in total we obtain at least
\begin{align}
\gamma t^{v(F)-2}\cdot \frac{1}{4}\xi_{\textrm{CL}}p^{e(F)-1}\left(\frac{n}{t}\right)^{v(F)-2}|W|^2
&=\frac{\gamma \xi_{\textrm{CL}}}{4}\cdot p^{e(F)-1}n^{v(F)-2}|W|^2 \nonumber\\
&\geq 
\frac{\gamma \xi_{\textrm{CL}}}{4}\cdot C_0^{e(F)-1}|W|^2
\label{eq:number Fe}
\end{align}
copies of~$F^\prime$ in~$H^\prime$ with $a_1\in W_1$ and $a_2\in W_2$.
For a pair of vertices $e\in \binom{W}{2}$ we define
\[
	x_e=\left\vert\left\lbrace\vphi(F^\prime) \text{ copy of } F^\prime \text{ in } H^\prime \colon e=\{\vphi(a_1),\vphi(a_2)\} \right\rbrace\right\vert\,.
\]
By~\eqref{eq:number Fe} we know that
\begin{equation}\label{eq:sum xe}
\sum_{e\in \binom{W}{2}} x_e
\geq
\frac{\gamma \xi_{\textrm{CL}}}{4}\cdot C_0^{e(F)-1}|W|^2\,.
\end{equation}

Let $\cW_{>0}=\left\lbrace e\in \binom{W}{2}:x_e\neq 0\right\rbrace$ and $N=\vert \cW_{>0}\vert$.
Since this~$N$ corresponds to the number of edges in~$\Base_{H^\prime}(F)[W]\subseteq \Base_{H}(F)[W]$ we shall show that~$N\geq d\binom{\vert W\vert}{2}$.
For this purpose we use~\eqref{eq:sum xe} and an upper bound for $\sum_{e\in\binom{W}{2}} x_e^2$ that follows from~\eqref{eq:number two Fe}. In fact, 

\begin{equation}
\sum_{e\in\binom{W}{2}}x_e^2 
\overset{\eqref{eq:number two Fe}}{\leq}  \vert\cF^{\star}\vert \cdot 2p^{e(\hat{F})}n^{v(\hat{F})-2}|W|^2\label{eq:sum xe^2 zwischenschritt}
\end{equation}
where~$\hat{F}$ is a graph in~$\cF^{\star}$ that maximises the value of~$p^{e(F^{\star})}n^{v(F^{\star})-2}$ for $F^{\star}\in\cF^{\star}$.
We will show that $p^{e(\hat{F})}n^{v(\hat{F})-2}$ is bounded by a constant only depending on $C_0$, $C_1$ and~$F$.
In fact, for~$F^\star=\vphi_1(F')\cup \vphi_2(F')\in \cF^\star$ let~$F_0=\vphi_1(F^\prime)\cap\vphi_2(F^\prime)$ and $e=\{\vphi_1(a_1),\vphi_1(a_2)\}$.
In particular, $F_0+e\subseteq F$ and we have

\begin{equation*}
	p^{e(F^\star)}n^{v(F^\star)-2}
	=
	\frac{p^{e(F^\star+e)}n^{v(F^\star+e)}}{pn^2}
	=
	\frac{\big(p^{e(F)}n^{v(F)}\big)^2}{p^{e(F_0+e)}n^{v(F_0+e)}\cdot pn^2}
	\overset{\eqref{eq:expF}}{\leq}
	\frac{C_1^{2e(F)-2}pn^2}{p^{e(F_0+e)}n^{v(F_0+e)}}
	\overset{\eqref{eq:minexp}}{\leq}\frac{C_1^{2e(F)-2}}{C'_0}\,.
\end{equation*}
Combining~\eqref{eq:sum xe^2 zwischenschritt} with the simple upper bound  $\vert\cF^\star\vert\leq (v(F)+1)^{v(F)}$
and the last inequality yields
\begin{equation}
\sum_{e\in\binom{W}{2}}x_e^2 \leq 2(v(F)+1)^{v(F)}\frac{C_1^{2(e(F)-1)}}{C'_0}|W|^2\,.\label{eq:sum xe^2}
\end{equation}

Finally, we establish the $(\rho,d)$-denseness of~$G_0$. In fact,
from the Cauchy-Schwarz inequality we know
\[
	\Bigg(\sum_{e\in\binom{W}{2}}x_e\Bigg)^2=\left(\sum_{e\in\cW_{>0}}x_e\right)^2
	\leq
	N\cdot\sum_{e\in \cW_{>0}}x_e^2
	=
	N\cdot\sum_{e\in \binom{W}{2}}x_e^2\,.
\]
and, consequently,
\begin{align*}
N
&\overset{\phantom{\eqref{eq:sum xe},\eqref{eq:sum xe^2}}}{\geq} 
\frac{\left(\sum_{e\in\binom{W}{2}}x_e\right)^2}{\sum_{e\in\binom{W}{2}}x_e^2}\\
&\overset{\eqref{eq:sum xe},\eqref{eq:sum xe^2}}{\geq} 
\frac{\left(\gamma \xi_{\textrm{CL}}C_0^{e(F)-1}|W|^2/4\right)^2}
{2(v(F)+1)^{v(F)}C_1^{2(e(F)-1)}|W|^2/C'_0}\\
&\overset{\phantom{\eqref{eq:sum xe},\eqref{eq:sum xe^2}}}{>}
\frac{\gamma^2 \xi_{\textrm{CL}}^2 C_0^{2(e(F)-1)}C'_0}{16 (v(F)+1)^{v(F)}C_1^{2(e(F)-1)}}\cdot \binom{|W|}{2}\\
&\overset{\eqref{eq:d},\eqref{eq:gamma}}{\geq}d\cdot \binom{|W|}{2}\,.
\end{align*}
Recalling that $W\subseteq V_{a_0}$ with $\vert W\vert\geq  \rho\vert V_{a_0}\vert$ was arbitrary, implies that~$G_0$ is $(\rho,d)$-dense which finishes the proof.
\end{proof}

\section{Concluding remarks}

\subsection{Ramsey properties for  \texorpdfstring{$\ZZn$}{the integers}}
The methods used here can be adjusted to obtain the sharpness for some cases of Rado's theorem for two colours
in $\ZZn$. For van der Waerden's theorem such a result appeared in~\cite{FHPS14} and, in fact, the work presented here 
relied on some of those ideas. However, the approach in~\cite{FHPS14} made use of the fact that the corresponding extremal problem 
(known as Szemer\'edi's theorem) has density~$0$, which limits the approach to so-called \emph{density regular systems} (see, e.g.,~\cite{FGR88}).
Maybe the simplest regular, but not density regular, instance of Rado's theorem is the well known result of Schur~\cite{Schur}, which asserts for 
finite colourings of $\ZZn$ the existence of  a
monochromatic solution for the equation $x+y=z$ for sufficiently large~$n$. The threshold for 
this property appeared in~\cite{GRR96} for two colours and in~\cites{CG10,FrRoSch10} for an arbitrary number of colours. 
The sharpness for two colours is based on some of the ideas used in~\cite{FHPS14} and the work here, will appear in the PhD thesis
of the second author~\cite{Schul}. 

\subsection{Ramsey properties of nearly partite hypergraphs}
Instead of nearly bipartite graphs one may consider \emph{nearly $k$-partite $k$-uniform hypergraphs}, i.e.,
$k$-uniform hypergraphs with vertex partition $V_1\dcup\dots\dcup V_k$ and the property 
that at most one hyperedge is contained in $V_1$ and the remaining hyperedges contain exactly one vertex from 
each vertex class. Again one may require additional \emph{balancedness} assumptions (similar as in Theorem~\ref{thm:Main}).
However, for the proof of a lemma corresponding to Lemma~\ref{lem:base} one would need a sparse version of the so-called
weak regularity lemma for hypergraphs and a corresponding embedding/counting lemma for subhypergraphs of random 
hypergraphs (see, e.g.,~\cite{CGSS14}*{Section~5.1}). For the more relaxed version of nearly partite, which would 
allow the additional hyperedge to span across more than one vertex class, one would likely need sparse analogues of 
the strong hypergraph regularity method for subhypergraphs of random hypergraphs.

\subsection{Ramsey properties for more general graphs and more colours}
It would be very interesting to extend Theorem~\ref{thm:Main} to more general graphs~$F$. The class of 
nearly bipartite graphs contains the triangle~$K_3$ and an extension for all cliques 
would be desirable. The main obstacle seems to establish a suitable analogue of Lemma~\ref{lem:base}
for this case.

Another limitation is the restriction to two colours only. The R\"odl-Ruci\'nski theorem~\cite{RR95}
applies, up to very few exceptions (see, e.g.,~\cite{JLR00}*{Section~8.1}), to arbitrary graphs and any number of colours~$r\geq 2$.
However, besides for the case of trees (see~\cite{FrKr00}), all known sharpness results address only the two-colour case 
and extending these results to more than two colours appears an interesting open problem in the area. 

Finally, we mention that 
due to Friedgut's criterion the $c=c(n)$ in  Theorem~\ref{thm:Main} is bounded by constants, but it may depend on~$n$.
It seems plausible, that  a strengthening of Theorem~\ref{thm:Main} for some constant $c$ independent of $n$
also holds. However, this would likely require a very different approach to these problems.

\subsection*{Acknowledgement} We are indebted to both referees for their thorough 
reading of the manuscript and their constructive and helpful remarks.


\begin{bibdiv}
\begin{biblist}

\bib{BMS12}{article}{
   author={Balogh, J{\'o}zsef},
   author={Morris, Robert},
   author={Samotij, Wojciech},
   title={Independent sets in hypergraphs},
   journal={J. Amer. Math. Soc.},
   volume={28},
   date={2015},
   number={3},
   pages={669--709},
   issn={0894-0347},
   review={\MR{3327533}},
   doi={10.1090/S0894-0347-2014-00816-X},
}

\bib{CG10}{article}{
   author={Conlon, D.},
   author={Gowers, W. T.},
   title={Combinatorial theorems in sparse random sets},
   journal={Ann. of Math. (2)},
   volume={184},
   date={2016},
   number={2},
   pages={367--454},
   issn={0003-486X},
   review={\MR{3548529}},
   doi={10.4007/annals.2016.184.2.2},
}

\bib{CGSS14}{article}{
   author={Conlon, D.},
   author={Gowers, W. T.},
   author={Samotij, W.},
   author={Schacht, M.},
   title={On the K\L R conjecture in random graphs},
   journal={Israel J. Math.},
   volume={203},
   date={2014},
   number={1},
   pages={535--580},
   issn={0021-2172},
   review={\MR{3273450}},
   doi={10.1007/s11856-014-1120-1},
}


\bib{ES84}{article}{
   author={Erd{\H{o}}s, P.},
   author={Simonovits, M.},
   title={Cube-supersaturated graphs and related problems},
   conference={
      title={Progress in graph theory},
      address={Waterloo, Ont.},
      date={1982},
   },
   book={
      publisher={Academic Press, Toronto, ON},
   },
   date={1984},
   pages={203--218},
   review={\MR{776802}},
}

\bib{FGR88}{article}{
   author={Frankl, P.},
   author={Graham, R. L.},
   author={R{\"o}dl, V.},
   title={Quantitative theorems for regular systems of equations},
   journal={J. Combin. Theory Ser. A},
   volume={47},
   date={1988},
   number={2},
   pages={246--261},
   issn={0097-3165},
   review={\MR{930955 (89d:05020)}},
   doi={10.1016/0097-3165(88)90020-9},
}

\bib{FrBou99}{article}{
   author={Friedgut, Ehud},
   title={Sharp thresholds of graph properties, and the $k$-sat problem},
   note={With an appendix by Jean Bourgain},
   journal={J. Amer. Math. Soc.},
   volume={12},
   date={1999},
   number={4},
   pages={1017--1054},
   issn={0894-0347},
   review={\MR{1678031 (2000a:05183)}},
   doi={10.1090/S0894-0347-99-00305-7},
}

\bib{Fr05}{article}{
   author={Friedgut, Ehud},
   title={Hunting for sharp thresholds},
   journal={Random Structures Algorithms},
   volume={26},
   date={2005},
   number={1-2},
   pages={37--51},
   issn={1042-9832},
   review={\MR{2116574 (2005k:05216)}},
   doi={10.1002/rsa.20042},
}

\bib{FHPS14}{article}{
    author={Friedgut, Ehud},
	author={H{\`a}n, H.},
	author={Person, Y.},
	author={Schacht, M.},
    title={A sharp threshold for van der {W}aerden's theorem in random subsets},
   journal={Discrete Anal.},
   date={2016},
   pages={Paper No. 615, 20},
   issn={2397-3129},
   review={\MR{3533306}},
   doi={10.19086/da.615},
}


\bib{FrKr00}{article}{
   author={Friedgut, Ehud},
   author={Krivelevich, Michael},
   title={Sharp thresholds for certain Ramsey properties of random graphs},
   journal={Random Structures Algorithms},
   volume={17},
   date={2000},
   number={1},
   pages={1--19},
   issn={1042-9832},
   review={\MR{1768845 (2001i:05136)}},
   doi={1r0.1002/1098-2418(200008)17:1$<$1::AID-RSA1$>$3.0.CO;2-4},
}

\bib{FRRT06}{article}{
   author={Friedgut, Ehud},
   author={R{\"o}dl, Vojtech},
   author={Ruci{\'n}ski, Andrzej},
   author={Tetali, Prasad},
   title={A sharp threshold for random graphs with a monochromatic triangle
   in every edge coloring},
   journal={Mem. Amer. Math. Soc.},
   volume={179},
   date={2006},
   number={845},
   pages={vi+66},
   issn={0065-9266},
   review={\MR{2183532 (2006h:05208)}},
   doi={10.1090/memo/0845},
}

\bib{FrRoSch10}{article}{
   author={Friedgut, Ehud},
   author={R{\"o}dl, Vojt{\v{e}}ch},
   author={Schacht, Mathias},
   title={Ramsey properties of random discrete structures},
   journal={Random Structures Algorithms},
   volume={37},
   date={2010},
   number={4},
   pages={407--436},
   issn={1042-9832},
   review={\MR{2760356 (2012a:05274)}},
   doi={10.1002/rsa.20352},
}

\bib{GRR96}{article}{
   author={Graham, Ronald L.},
   author={R{\"o}dl, Vojtech},
   author={Ruci{\'n}ski, Andrzej},
   title={On Schur properties of random subsets of integers},
   journal={J. Number Theory},
   volume={61},
   date={1996},
   number={2},
   pages={388--408},
   issn={0022-314X},
   review={\MR{1423060 (98c:05013)}},
   doi={10.1006/jnth.1996.0155},
}

\bib{Ja90}{article}{
   author={Janson, Svante},
   title={Poisson approximation for large deviations},
   journal={Random Structures Algorithms},
   volume={1},
   date={1990},
   number={2},
   pages={221--229},
   issn={1042-9832},
   review={\MR{1138428 (93a:60041)}},
   doi={10.1002/rsa.3240010209},
}
		

\bib{JLR90}{article}{
   author={Janson, Svante},
   author={{\L}uczak, Tomasz},
   author={Ruci{\'n}ski, Andrzej},
   title={An exponential bound for the probability of nonexistence of a
   specified subgraph in a random graph},
   conference={
      title={Random graphs '87},
      address={Pozna\'n},
      date={1987},
   },
   book={
      publisher={Wiley, Chichester},
   },
   date={1990},
   pages={73--87},
   review={\MR{1094125 (91m:05168)}},
}
		
\bib{JLR00}{book}{
   author={Janson, Svante},
   author={{\L}uczak, Tomasz},
   author={Ruci{\'n}ski, Andrzej},
   title={Random graphs},
   series={Wiley-Interscience Series in Discrete Mathematics and
   Optimization},
   publisher={Wiley-Interscience, New York},
   date={2000},
   pages={xii+333},
   isbn={0-471-17541-2},
   review={\MR{1782847 (2001k:05180)}},
   doi={10.1002/9781118032718},
}

\bib{Ko97}{article}{
   author={Kohayakawa, Y.},
   title={Szemer\'edi's regularity lemma for sparse graphs},
   conference={
      title={Foundations of computational mathematics},
      address={Rio de Janeiro},
      date={1997},
   },
   book={
      publisher={Springer, Berlin},
   },
   date={1997},
   pages={216--230},
   review={\MR{1661982 (99g:05145)}},
}

\bib{KST54}{article}{
   author={K{\H{o}}v{\'a}ri, T.},
   author={S{\'o}s, V. T.},
   author={Tur{\'a}n, P.},
   title={On a problem of K. Zarankiewicz},
   journal={Colloquium Math.},
   volume={3},
   date={1954},
   pages={50--57},
   review={\MR{0065617 (16,456a)}},
}

\bib{RR93}{article}{
   author={R{\"o}dl, Vojt{\v{e}}ch},
   author={Ruci{\'n}ski, Andrzej},
   title={Lower bounds on probability thresholds for Ramsey properties},
   conference={
      title={Combinatorics, Paul Erd\H os is eighty, Vol.\ 1},
   },
   book={
      series={Bolyai Soc. Math. Stud.},
      publisher={J\'anos Bolyai Math. Soc., Budapest},
   },
   date={1993},
   pages={317--346},
   review={\MR{1249720 (95b:05150)}},
}

\bib{RR95}{article}{
   author={R{\"o}dl, Vojt{\v{e}}ch},
   author={Ruci{\'n}ski, Andrzej},
   title={Threshold functions for Ramsey properties},
   journal={J. Amer. Math. Soc.},
   volume={8},
   date={1995},
   number={4},
   pages={917--942},
   issn={0894-0347},
   review={\MR{1276825 (96h:05141)}},
   doi={10.2307/2152833},
}

\bib{ST12}{article}{
   author={Saxton, David},
   author={Thomason, Andrew},
   title={Hypergraph containers},
   journal={Invent. Math.},
   volume={201},
   date={2015},
   number={3},
   pages={925--992},
   issn={0020-9910},
   review={\MR{3385638}},
   doi={10.1007/s00222-014-0562-8},
}

\bib{Sch09}{article}{
   author={Schacht, Mathias},
   title={Extremal results for random discrete structures},
   journal={Ann. of Math. (2)},
   volume={184},
   date={2016},
   number={2},
   pages={333--365},
   issn={0003-486X},
   review={\MR{3548528}},
   doi={10.4007/annals.2016.184.2.1},
}

\bib{Schul}{thesis}{
   author={Schulenburg, Fabian},
   title={Threshold results for cycles},
   type={PhD thesis},
   organization={Fachbereich Mathematik, Universit\"at Hamburg},
   note={Submitted},
}

\bib{Schur}{article}{
    Author = {I. {Schur}},
    Title = {{\"Uber die Kongruenz $x^m+y^m\equiv z^m (\text{mod}\ p)$}},
    Journal = {{Jahresber. Dtsch. Math.-Ver.}},
    ISSN = {0012-0456; 1869-7135/e},
    Volume = {25},
    Pages = {114--117},
    date = {1916},
    note = {German},
}

\bib{Sp90}{article}{
   author={Spencer, Joel},
   title={Counting extensions},
   journal={J. Combin. Theory Ser. A},
   volume={55},
   date={1990},
   number={2},
   pages={247--255},
   issn={0097-3165},
   review={\MR{1075710 (91j:05096)}},
   doi={10.1016/0097-3165(90)90070-D},
}

\end{biblist}
\end{bibdiv}
\end{document}